\documentclass[11pt]{amsart}

\usepackage{amsmath,enumitem,graphicx,url,hhline,diagbox}

\usepackage[mathlines]{lineno}

\usepackage[colorlinks=false]{hyperref}

\usepackage{cleveref}

\theoremstyle{plain}
\newtheorem{theorem}{Theorem}[section]

\newtheorem{proposition}[theorem]{Proposition}
\newtheorem{corollary}[theorem]{Corollary}
\newtheorem{sublemma}{}[theorem]

\newcommand{\dash}{\nobreakdash-\hspace{0pt}}
\newcommand{\ba}{\backslash}
\newcommand{\ifc}{internally $4$\dash connected}

\newcommand{\ffsc}{$(4,4,S)$\dash connected}
\newcommand{\cml}[1]{\ensuremath{\mathit{CM}_{\hspace{-1.5pt}#1}}}
\newcommand{\qml}[1]{\ensuremath{\mathit{QM}_{\hspace{-1.5pt}#1}}}

\title[
A splitter theorem for binary matroids]
{Towards a splitter theorem for internally $4$-connected
binary matroids VIII: small matroids}

\author[Chun]{Carolyn Chun}
\address{Department of Mathematics,
United States Naval Academy,
Annapolis, MD, 21402 USA}
\email{chun@usna.edu}

\author[Mayhew]{Dillon Mayhew}
\address{School of Mathematics and Statistics,
Victoria University of Wellington,
New Zealand}
\email{dillon.mayhew@vuw.ac.nz}

\author[Oxley]{James Oxley}
\address{Mathematics Department,
Louisiana State University,
United States of America}
\email{oxley@math.lsu.edu}

\date{\today}

\begin{document}

\begin{abstract}
Our splitter theorem for internally $4$\dash connected
binary matroids studies pairs of the form $(M,N)$,
where $M$ and $N$ are internally $4$\dash connected
binary matroids, $M$ has a proper $N$\dash minor,
and if $M'$ is an internally $4$\dash connected matroid
such that $M$ has a proper $M'$\dash minor and $M'$
has an $N$\dash minor, then $|E(M)|-|E(M')|>3$.
The analysis in the splitter theorem requires the
constraint that $|E(M)|\geq 16$.
In this article, we complement that analysis by
using an exhaustive computer search to find all such pairs
for which $|E(M)|\leq 15$.
\end{abstract}

\maketitle

\section{Introduction}
\label{sect1}

A matroid is \emph{\ifc} if it is $3$\dash connected and
$\min\{|X|,|Y|\}=3$ for any $3$\dash separation, $(X,Y)$.
For some time, we have been engaged in a project to develop
a splitter theorem for \ifc\ binary matroids
\cite{CMO12,CMO14a,CMO13,CMO14b,CMO14c,CO,COVII,CMOIX}.
This means that we are concerned with understanding
what we refer to here as interesting pairs.
If $N$ and $M$ are matroids, we write $N\preceq M$ to mean that
$M$ has an $N$\dash minor, and $N\prec M$ to mean that $M$
has a proper $N$\dash minor.
An \emph{interesting pair} is a pair $(M,N)$, where $M$ and $N$ are
\ifc\ binary matroids such that
\begin{itemize}
\item $|E(N)|\geq 6$;
\item $N\prec M$;
\item if $M'$ is an \ifc\ matroid
for which $N\preceq  M'\prec M$, then $|E(M)|-|E(M')|>3$.
\end{itemize}

Note that the last condition means that $|E(M)|-|E(N)|>3$.
We say that an interesting pair, $(M,N)$, is a
\emph{fascinating pair} if $M'$ is isomorphic to $N$
whenever $M'$ is an \ifc\ matroid
satisfying $N\preceq M'\prec M$.
Thus an interesting pair is fascinating if there is no intermediate
\ifc\ matroid in the minor order.

It has been known for some time (see, for example, \cite{JT02}) that there
are fascinating pairs with $|E(M)|-|E(N)|$ arbitrarily large;
indeed, this is true even if we insist that $M$ and $N$ are graphic matroids,
since we can produce a fascinating pair by setting $N$ to be the graphic
matroid of a cubic planar ladder, and letting $M$ be the graphic matroid of a
quartic planar ladder on the same number of vertices.
However, our project has shown that only a small number
of constructions are needed to build $M$ from $N$,
whenever $(M,N)$ is a fascinating pair.

The analysis in our project requires $M$ to be a certain
size, in particular, $|E(M)|\geq 16$.
To complement this analysis, our results here
contain a description of all interesting
pairs for which $|E(M)|\leq 15$.
Our first theorem will describe the fascinating pairs.
Up to duality, there are exactly $31$.
Before that, we introduce some important matroids and graphs.

For $n\geq 3$, we denote the \emph{cubic M\"{o}bius ladder} on $2n$
vertices by \cml{2n}.
This graph is obtained from an even cycle on $2n$ vertices
by joining each vertex to the antipodal vertex (the vertex of distance $n$).
Similarly, for $n\geq 2$, the \emph{quartic M\"{o}bius ladder} on $2n+1$ vertices
is denoted by \qml{2n+1}, and is obtained from an odd cycle with $2n+1$
vertices by joining each vertex to the two vertices of distance $n$.
Note that \qml{5} is isomorphic to $K_{5}$, and
\cml{6} is isomorphic to $K_{3,3}$.

The M\"obius matroids have been discovered in several contexts
\cite{kinglem,hardy}.
For each positive integer $n\ge 3$, let $\mathcal{W}_{n}$ be the
wheel with $n+1$ vertices, and let $B$ be the set of spoke edges.
Thus $B$ is a basis of the rank\dash $n$ binary matroid
$M({\mathcal W}_n)$.
Let $M_{n}$ be the binary matroid 
obtained from $M({\mathcal W}_n)$ by adding a single element,
$\gamma$, so that the fundamental circuit $C(\gamma, B)$ is
$B \cup \gamma$.
Kingan and Lemos~\cite{kinglem} denote $M_n$ by $F_{2n+1}$. 
Observe that $M_3$ is the Fano matroid, and $M_4 \cong M^*(K_{3,3})$.
When $n$ is odd, $M_n^*$ is the
\emph{rank-$(n+1)$ triadic M\"{o}bius matroid},
denoted by $\Upsilon_{n+1}$.
Hence $\Upsilon_4\cong F_{7}^{*}$.
Moreover, $\Upsilon_{6}$ is isomorphic to any
single-element deletion of $T_{12}$,
the rank\dash $6$ binary matroid
introduced by Kingan \cite{Kin97}.
We also observe that $\Upsilon_{n+1}\ba \gamma\cong M^{*}(\qml{n})$.

For $n \ge 3$, we construct the graph $G_{n+2}^{+}$ 
by starting with an $n$\dash vertex cycle, $C$, containing
adjacent vertices $x$ and $y$, and then adding two additional vertices,
$u$ and $w$, and making both of them adjacent to every vertex in $C$.
We join $u$ and $w$ with an edge $\gamma$.
Note that the planar dual of $G_{n+2}^{+}\ba \gamma$ is \cml{2n}.
Let $\Delta_{n+1}$ be the binary matroid that is obtained from 
$M(G_{n+2}^+)$ by deleting the element $xy$ and adding
a new element so that it forms a circuit with the elements $wx$ and $uy$.
This new element also forms a circuit with $ux$ and $wy$.
We also define $\Delta_{3}$ to be $F_{7}$.
Then $\Delta_{r}$ is the \emph{rank-$r$ triangular M\"{o}bius matroid}.
Observe that $\Delta_{n+1} \ba \gamma \cong M^{*}(\cml{2n})$.
Kingan and Lemos~\cite{kinglem} use $B_{3n+1}$ to denote $G_{n+2}^{+}$,
and $S_{3n+1}$ to denote $\Delta_{n+1}$.

Now we give our description of fascinating pairs.
Any graphs or matroids which we have not yet defined
will be introduced in \Cref{sect2}.
For now, we note that $Q_{3}$ is the cube graph;
$O$ is the octahedron graph;
$H_{1}$, $H_{2}$, and $H_{3}$ are graphs with $13$
edges, and, respectively, $6$, $7$, and $8$ vertices;
$Q_{3}^{\times}$ and $Y_{9}$ have $14$ edges and,
respectively, $8$ and $9$ vertices;
$A_{1}$, $A_{2}$, $A_{3}$, $A_{4}$, and $A_{5}$
are non-graphic matroids with rank $8$ and $14$ elements,
whereas $A_{6}$ has rank $7$ and $14$ elements;
the matroids $P$ and $Q$ have rank $4$ and $11$ elements;
each matroid of the form $B_{i}$ or $C_{j}$ has rank $8$ and $15$ elements;
both $R$ and $S$ have rank $5$ and $11$ elements,
while $D_{1}$ and $E_{1}$ have rank $9$ and $15$ elements.

\begin{theorem}
\label{maintheorem}
Assume that $(M_{0},N_{0})$ is a fascinating pair and $|E(M_{0})|\leq 15$.
Then, for some pair, $(M,N)$ in $\{(M_{0},N_{0}),(M_{0}^{*},N_{0}^{*})\}$,
one of the following statements holds.
\begin{enumerate}[label=\textup{(\arabic*)}]
\item $M$ is one of $M(Q_{3})$ or $M(K_{5})\cong M(\qml{5})$, and $N$ is $M(K_{4})$;
\item $M$ is one of $\Upsilon_{6}$ or $\Upsilon_{6}^{*}$, and $N$ is $F_{7}\cong \Upsilon_{4}^{*}$;
\item $M$ is one of $M(H_{1})$, $M(H_{2})$, $M(H_{3})$, or $M(\qml{7})$,
and $N$ is $M(K_{3,3})\cong M(\cml{6})$;
\item $M$ is one of $M(Q_{3}^{\times})$, $M(Y_{9})$, $M(\qml{7})$, or $M(\cml{10})$,
and $N$ is $M(K_{5})\cong M(\qml{5})$;
\item $M$ is one of $A_{1}$, $A_{2}$, $A_{3}$, $A_{4}$, $A_{5}$,
$A_{6}$, or $\Upsilon_{8}$, and $N$ is $\Delta_{4}$;
\item $M$ is one of $B_{1}$, $B_{2}$, $B_{3}$, $B_{4}$, or $B_{5}$,
and $N$ is $P$;
\item $M$ is one of $C_{1}$, $C_{2}$, $C_{3}$, or $C_{4}$, and $N$ is $Q$;
\item $(M,N)=(D_{1},R)$;
\item $(M,N)=(E_{1},S)$; or
\item $(M,N)=(\Upsilon_{8},\Upsilon_{6})$.
\end{enumerate}
\end{theorem}

With \Cref{maintheorem} in hand, it is easy to find the
pairs that are interesting but not fascinating:
there are only three (up to duality).

\begin{theorem}
\label{mtheorem2}
Assume that $(M_{0},N_{0})$ is an interesting pair that is not
fascinating and that $|E(M_{0})|\leq 15$.
Then there is a pair, $(M,N)$ in $\{(M_{0},N_{0}),(M_{0}^{*},N_{0}^{*},)\}$,
such that $(M,N)$ is either
$(M(\qml{7}),M(K_{4}))$,
$(\Upsilon_{8},F_{7},)$, or
$(\Upsilon_{8}^{*},F_{7})$.
\end{theorem}

The following table shows the number of interesting pairs
(up to duality), where the larger matroid has $m$ elements
in its ground set, and the smaller has $n$ elements.
Note that none of the pairs we have listed 
contains two self-dual matroids, so if we were not taking duality
into account, we would just double the numbers in the table.

\begin{center}
\begin{tabular}{|c||c|c|c|c|c|c|}\hline
\diagbox{$n$}{$m$}&10&11&12&13&14 &15\\\hhline{|=|=|=|=|=|=|=|}
6&1  &    &1  &    &1  &\\\hline
7&   & 2  &    &    &   & 2\\\hline
8&   &   &    &    &   & \\\hline
9&   &   &    &   3 & 1  & \\\hline
10&   &   &    &    & 9  &2\\\hline
11&   &   &    &    &   &12\\\hline
\end{tabular}
\end{center}

Next we note the specialisation of our theorems
to graphic matroids.
Any graphs not already defined are described in
\Cref{sect2}.
Let $G$ be a simple, $3$\dash connected graph.
For any partition, $(X,Y)$, of the edge set,
let $V(X,Y)$ be the set of vertices incident with
edges in both $X$ and $Y$.
We say that $G$ is \emph{\ifc}
if, whenever $3\leq |X|\leq |Y|$ we have that
$|V(X,Y)|\geq 3$, with equality implying that
$X$ is either a triangle or the set of edges
incident with a vertex of degree $3$.
In other words, $G$ is \ifc\
if and only if $M(G)$ is an \ifc\
matroid.

\begin{theorem}
\label{maingraph}
Assume $G_{1}$ and $G_{2}$ are
\ifc\ graphs such that
$|E(G_{1})|\leq 15$,
and $G_{1}$ has a proper $G_{2}$\dash minor.
Assume also that if $G$ is
an \ifc\ graph such that
$G_{1}$ has a proper $G$\dash minor, and
$G$ has a $G_{2}$\dash minor, then
$|E(G_{1})|-|E(G)|>3$.
Then one of the following statements holds.
\begin{itemize}
\item $G_{1}$ is one of $K_{5}$, $Q_{3}$, $O$, or $\qml{7}$,
and $G_{2}$ is $K_{4}$;
\item $G_{1}$ is one of $H_{1}$, $H_{2}$, $H_{3}$, or \qml{7},
and $G_{2}$ is $K_{3,3}$;
\item $G_{1}$ is one of $Q_{3}^{\times}$, $Y_{9}$, \qml{7}, or \cml{10},
and $G_{2}$ is $K_{5}$.
\end{itemize}
\end{theorem}

In many of the pairs in \Cref{maintheorem}
or \Cref{mtheorem2}, we encounter structures that are familiar from
the analysis in the rest of the project.
These structures lead to operations that we can use to
produce a smaller \ifc\ matroid from a larger one.
Four such operations will be documented in \Cref{moves}.
In the following results, we explain exactly when it is possible to
perform them on our fascinating and interesting pairs.

\begin{theorem}
\label{goodmovefirst}
Let the pair $(M,N)$ be as described in one of the statements
\textup{(1)}--\textup{(10)} in \textup{\Cref{maintheorem}}.
If $(M,N)$ is not one of $(M(Q_3),M(K_4))$, $(M(K_5),M(K_4))$,
$(\Upsilon_6, F_{7})$, $(\Upsilon_{6}^{*}, F_{7})$, $(M(\qml{7}),M(K_{3,3}))$,
or $(\Upsilon_8,\Delta_{4})$,
then $N$ can be obtained from $M$
(or $N^{*}$ can be obtained from $M^{*}$) by one of the
following four operations:
\begin{enumerate}[label=\textup{(\arabic*)}]
\item trimming a ring of bowties;
\item deleting the central cocircuit of a good augmented $4$\dash wheel;
\item a ladder-compression move; or
\item trimming an open rotor chain.
\end{enumerate}
\end{theorem}

The next \namecref{cor2} deals with the three interesting
pairs identified in \Cref{mtheorem2}.

\begin{corollary}
\label{cor2}
Let $(M,N)$ be
$(M(\qml{7}),M(K_{4}))$,
$(\Upsilon_{8},F_{7})$, or
$(\Upsilon_{8}^{*},F_{7})$.
Then there is an \ifc\ binary matroid, $M_{0}$, such that
$N\prec M_{0}\prec M$, and either $M_{0}$ can be obtained from $M$
(or $M_{0}^{*}$ can be obtained from $M^{*}$)
by a ladder-compression move.
\end{corollary}

We note that some of the exceptional pairs in \Cref{goodmovefirst}
are dealt with by some of the specific scenarios from our main
theorem, which appears in \cite{CMOIX}.
In particular, since $\Delta_{3}\cong F_{7}$, we see that if
$(M,N)$ is $(\Upsilon_6, F_{7})$ or $(\Upsilon_8,\Delta_{4})$, then
$M$ is a triadic M\"{o}bius matroid of rank $2r$, and
$N$ is a triangular M\"{o}bius matroid of rank $r$.
If $(M,N)$ is $(M(\qml{7}),M(\cml{6}))$, then $M$ is the
cycle matroid of a quartic M\"{o}bius ladder, and
$N$ is the cycle matroid of a cubic M\"{o}bius ladder,
and $r(N)=r(M)-1$.
Thus the only truly exceptional pairs are $(M(K_5),M(K_4))$,
$(M(Q_3),M(K_4))$, and $(\Upsilon_{6}^{*}, F_{7})$.

We prove \Cref{maintheorem,mtheorem2} with an exhaustive search,
using the matroid functionality of the sage mathematics package,
(Version $6.10$) \cite{Stein}.
All the computations performed in this search were
performed on a single desktop computer, and took
a total of approximately $55$ hours of computation.
The code used in the search is available from
\url{http://homepages.ecs.vuw.ac.nz/~mayhew/splittertheorem.shtml}.
Some of the objects created during the search,
such as the catalogue of $3$\dash connected binary
matroids with at most $15$ elements,
required a non-trivial amount of computation.
Those objects are also available at the same site.

\section{Winning Moves}
\label{moves}

In this section, we describe four different structures that appear
naturally when we examine \ifc\ binary matroids.
Each structure allows us to perform certain deletions and contractions
to obtain an \ifc\ proper minor.
These operations play an essential role in the statement of our splitter
theorem.
In \Cref{sect2}, we analyse the pairs in \Cref{maintheorem,mtheorem2},
and demonstrate that, in many cases, these structures appear there also.

Recall that a $3$\dash connected matroid is \ifc\
when every $3$\dash separation has a triangle or a triad on one side.
A $4$\dash element fan is a set $\{x_1,x_2,x_3,x_4\}$,
where $\{x_1,x_2,x_3\}$ is a triangle and $\{x_2,x_3,x_4\}$ is a triad.
A $3$\dash connected matroid, $M$, is
\emph{$(4,4,S)$\dash connected} if, for every $3$\dash separation,
$(X,Y)$, of $M$, one of $X$ and $Y$ is a triangle, a triad, or a
$4$\dash element fan.

\begin{figure}[htb]
\centering
\includegraphics[scale=1.1]{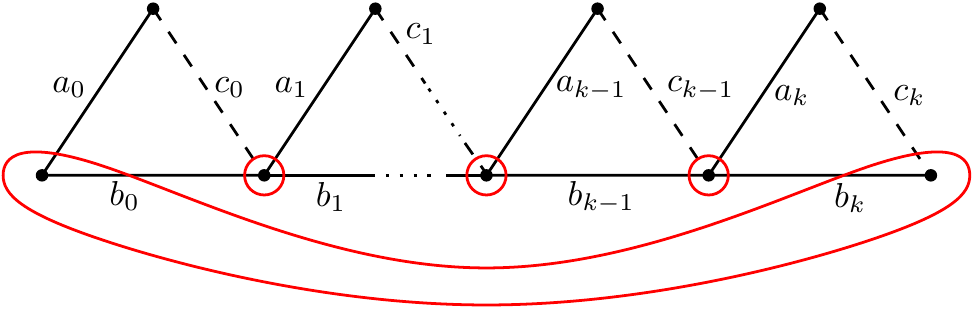}
\caption{A bowtie ring.  All elements are distinct.}
\label{btringfig}
\end{figure}

A \emph{bowtie} consists of a pair of disjoint triangles
whose union contains a $4$\dash element cocircuit.
Assume $k\geq 2$, and $T_0,T_1, \dots,T_k$ is a sequence of
pairwise disjoint triangles.
Let $T_i$ be $\{a_i,b_i,c_i\}$ for $i\in \{0,1,\ldots, k\}$.
Assume $D_i=\{b_i,c_i,a_{i+1},b_{i+1}\}$ is
a cocircuit for $i\in \{0,1,\ldots, k-1\}$, and in addition,
$D_{k}=\{b_{k},c_{k},a_{0},b_{0}\}$ is a cocircuit.
Then we say that $T_0,D_0,T_1,D_1,\dots ,T_k,D_k$ is a \emph{ring of bowties}.
Although the matroid $M$ we are dealing with need not be graphic, we follow the
convention begun in \cite{CMO11} of using a modified graph diagram to keep
track of some of the circuits and cocircuits in $M$.   
\Cref{btringfig} shows  such a modified graph diagram.  
Each of the cycles in such a graph diagram corresponds to a circuit  of $M$
while a circled vertex indicates a known cocircuit of $M$.
If $M'=M\ba\{c_0,c_1,\dots ,c_k\}$, then we say that $M'$ is obtained from
$M$ by \emph{trimming a ring of bowties}.  

\begin{figure}[htb]
\centering
\includegraphics[scale=1.1]{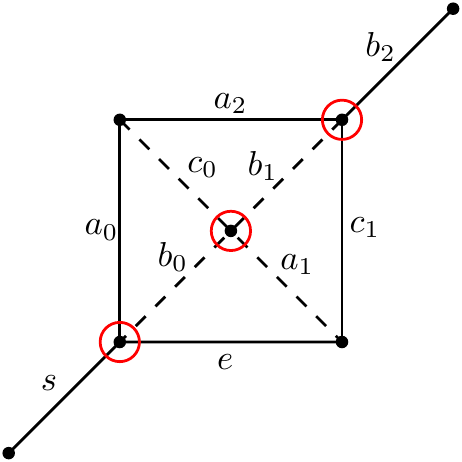}
\caption{An augmented $4$-wheel.  All displayed elements are distinct.}
\label{a4w}
\end{figure}

An \emph{augmented $4$\dash wheel} is represented by the modified graph
diagram in \Cref{a4w}, where the four dashed edges form the
\emph{central cocircuit}. 
 If a matroid $M$ contains the structure in \Cref{a4w} and $M\ba e$ is
\ffsc, then we say that the augmented $4$\dash wheel is \emph{good}.
We refer to the operation of deleting the four dashed edges as
\emph{removing the central cocircuit of an augmented $4$\dash wheel}.  

\begin{figure}[htb]
\centering
\includegraphics[scale=1.1]{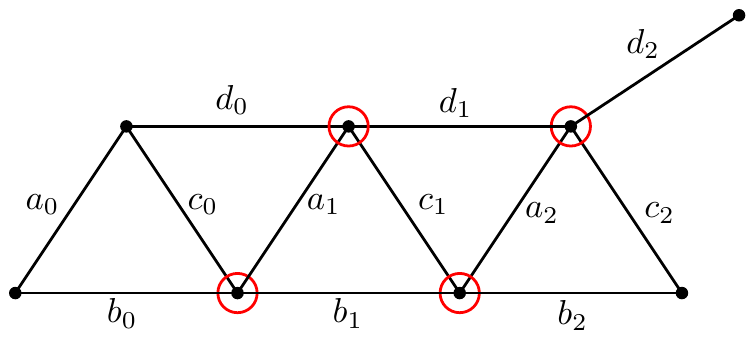}
\caption{The elements shown are distinct.}
\label{laddery}
\end{figure}

Our third structure requires a special four-element move.  
If $M$ contains the structure in \Cref{laddery},
then we say that $M\ba c_{1},c_2 /d_{1},b_2$ is obtained from
$M$ by a \emph{ladder-compression move}.  

\begin{figure}[htb]
\centering
\includegraphics[scale=1.1]{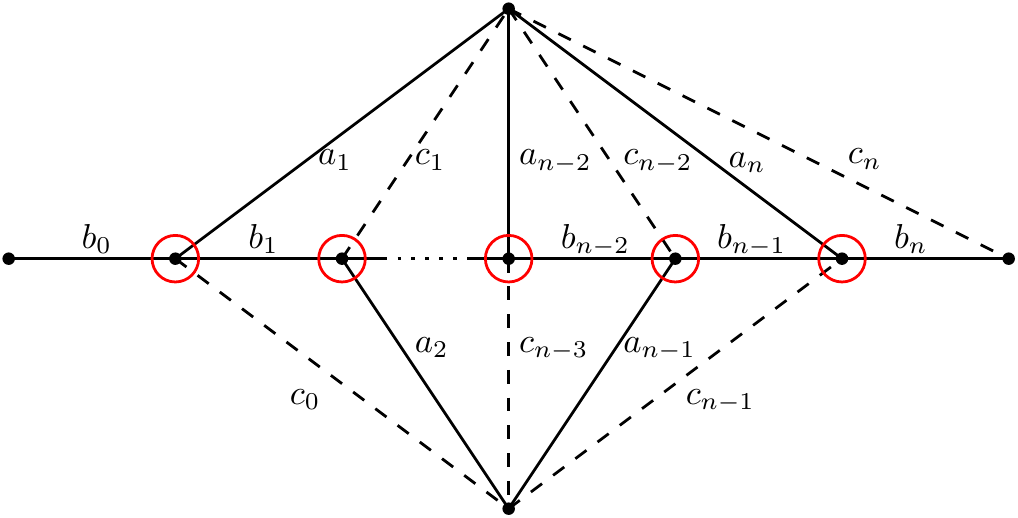}
\caption{All the elements shown are distinct.}
\label{gmcdashed}
\end{figure}

Finally, we consider the structure in \Cref{gmcdashed}.
Note that $n$ may be either even or odd.
When there are at least three dashed elements, we refer to
the structure in \Cref{gmcdashed} as an \emph{open rotor chain}
and we refer to the operation of deleting the dashed elements as
\emph{trimming an open rotor chain}.

\section{The special graphs and matroids}
\label{sect2}

This section has two purposes.
First, we introduce all the graphs and matroids
that feature in \Cref{maintheorem}, which we now restate.

\begin{theorem}
\label{repeat}
Assume that $(M_{0},N_{0})$ is a fascinating pair and $|E(M_{0})|\leq 15$.
Then, for some pair, $(M,N)$ in $\{(M_{0},N_{0}),(M_{0}^{*},N_{0}^{*})\}$,
one of the following statements holds.
\begin{enumerate}[label=\textup{(\arabic*)}]
\item $M$ is one of $M(Q_{3})$ or $M(K_{5})\cong M(\qml{5})$, and $N$ is $M(K_{4})$;
\item $M$ is one of $\Upsilon_{6}$ or $\Upsilon_{6}^{*}$, and $N$ is $F_{7}\cong \Upsilon_{4}^{*}$;
\item $M$ is one of $M(H_{1})$, $M(H_{2})$, $M(H_{3})$, or $M(\qml{7})$,
and $N$ is $M(K_{3,3})\cong M(\cml{6})$;
\item $M$ is one of $M(Q_{3}^{\times})$, $M(Y_{9})$, $M(\qml{7})$, or $M(\cml{10})$,
and $N$ is $M(K_{5})\cong M(\qml{5})$;
\item $M$ is one of $A_{1}$, $A_{2}$, $A_{3}$, $A_{4}$, $A_{5}$,
$A_{6}$, or $\Upsilon_{8}$, and $N$ is $\Delta_{4}$;
\item $M$ is one of $B_{1}$, $B_{2}$, $B_{3}$, $B_{4}$, or $B_{5}$,
and $N$ is $P$;
\item $M$ is one of $C_{1}$, $C_{2}$, $C_{3}$, or $C_{4}$, and $N$ is $Q$;
\item $(M,N)=(D_{1},R)$;
\item $(M,N)=(E_{1},S)$; or
\item $(M,N)=(\Upsilon_{8},\Upsilon_{6})$.
\end{enumerate}
\end{theorem}

In many of the pairs from this theorem, it is possible to apply one
of the four moves described in \Cref{moves}.
Thus the second purpose of this section is to document these moves, and
ultimately prove \Cref{goodmovefirst},
which we restate next.

\begin{theorem}
\label{goodmove}
Let the pair $(M,N)$ be as described in one of the statements
\textup{(1)}--\textup{(10)} in \textup{\Cref{maintheorem}}.
If $(M,N)$ is not one of $(M(Q_3),M(K_4))$, $(M(K_5),M(K_4))$,
$(\Upsilon_6, F_{7})$, $(\Upsilon_{6}^{*}, F_{7})$, $(M(\qml{7}),M(K_{3,3}))$,
or $(\Upsilon_8,\Delta_{4})$,
then $N$ can be obtained from $M$
(or $N^{*}$ can be obtained from $M^{*}$) by one of the
following four operations:
\begin{enumerate}[label=\textup{(\arabic*)}]
\item trimming a ring of bowties;
\item deleting the central cocircuit of a good augmented $4$\dash wheel;
\item a ladder-compression move; or
\item trimming an open rotor chain.
\end{enumerate}
\end{theorem}

Now we start describing various graphs and matroids, beginning with
the graphs $K_{4}$, $K_{5}$, and $Q_{3}$,
all of which are illustrated in \Cref{fig16}.
The graph $Q_{3}$ is also known as the \emph{cube graph}.
\Cref{fig16} also shows the \emph{octahedron graph},
$O$, which is the planar dual of $Q_{3}$.

\begin{figure}[htb]
\centering
\includegraphics[scale=1.1]{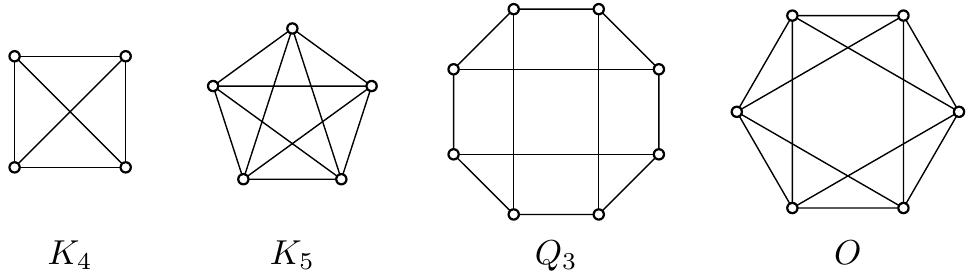}
\caption{Graphs $K_{4}$, $K_{5}$, $Q_{3}$, and $O$.}
\label{fig16}
\end{figure}

In Lemma~2.3 of \cite{GZ06}, Geelen and Zhou describe
five \ifc\ graphs having $K_{3,3}\cong \cml{6}$ as a minor.
One of the five is \cml{8}, which has only $12$ edges.
Another is isomorphic to \qml{7}.
Let the other three graphs be $H_{1}$, $H_{2}$, and $H_{3}$.
These are shown in \Cref{fig17}.

\begin{figure}[htb]
\centering
\includegraphics[scale=1.1]{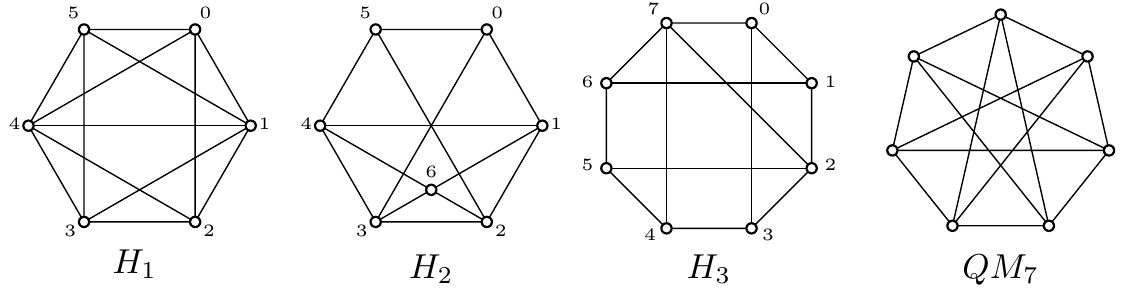}
\caption{Graphs $H_{1}$, $H_{2}$, $H_{3}$, and \qml{7}.}
\label{fig17}
\end{figure}

\begin{proposition}
\label{prop2}
Let $(M,N)$ be one of the pairs
$(M(H_{1}),M(K_{3,3}))$,
$(M(H_{2}),M(K_{3,3}))$, or
$(M^{*}(H_{3}),M^{*}(K_{3,3}))$.
Then $N$ is obtained from $M$ by trimming a bowtie ring,
deleting the central cocircuit from a good augmented $4$\dash wheel,
or a ladder-compression move.
\end{proposition}

\begin{proof}
Note that $M(H_1)$ has the bowtie ring shown in \Cref{btring_h1},
and trimming this ring yields $M(K_{3,3})$.
Also, $M(H_2)$ has a good augmented $4$\dash wheel
whose central cocircuit is the set of edges incident with vertex $6$.  
Deleting this cocircuit yields $M(K_{3,3})$.
Finally, $M^{*}(H_3)$ has the ladder segment shown
in \Cref{laddery}, where edges
$(16,12,01,07,03,23,34,47,45,25,56,67)$ correspond to
$(a_0,b_0,c_0,d_0,a_1,b_1,c_1,d_1,a_2,b_2,c_2,d_2)$.
If we delete $c_1$ and $c_2$, and contract
$d_1$ and $b_2$, then we obtain $M^{*}(K_{3,3})$.
\end{proof}

\begin{figure}[htb]
\centering
\includegraphics[scale=1.1]{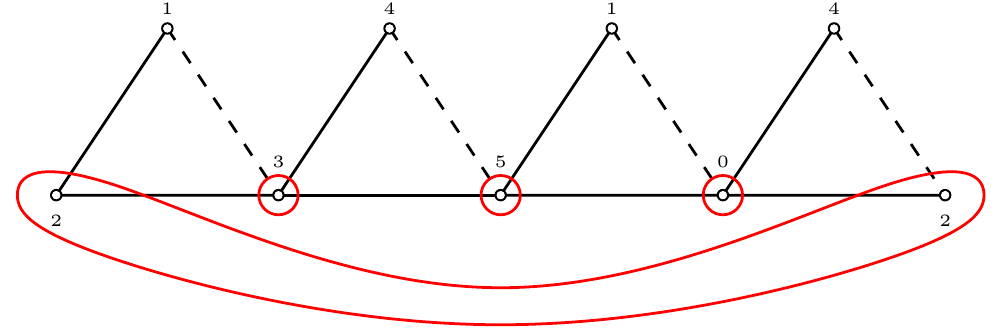}
\caption{Bowtie ring in $H_1$.}
\label{btring_h1}
\end{figure}

Observe that of all the pairs in statements (1), (2), and (3)
in \Cref{repeat} are either exceptional pairs that appear in
\Cref{goodmove}, or are dealt with by \Cref{prop2}.
Thus we have verified \Cref{goodmove} for these pairs.

The graphs $Q_{3}^{\times}$ and $Y_{9}$ are shown in \Cref{fig18},
along with \cml{10}.

\begin{figure}[htb]
\centering
\includegraphics[scale=1.1]{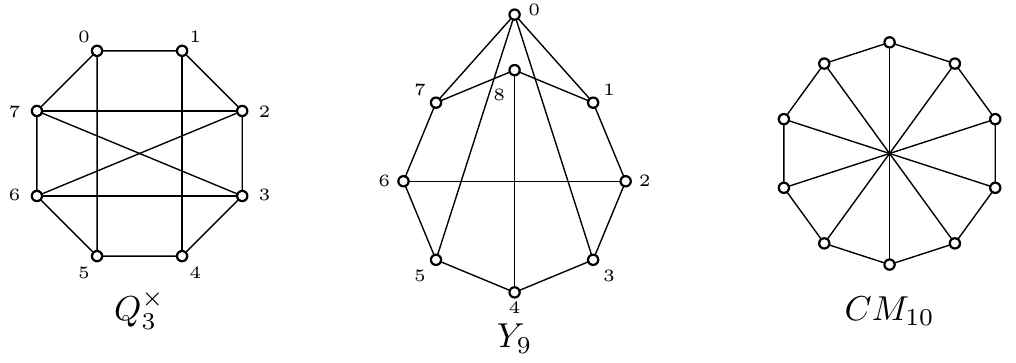}
\caption{Graphs $Q_{3}^{\times}$, $Y_{9}$, and \cml{10}.}
\label{fig18}
\end{figure}

\begin{proposition}
\label{prop3}
Let $(M,N)$ be one of the pairs
$(M^{*}(Q_{3}^{\times}),M^{*}(K_{5}))$,
$(M^{*}(Y_{9}),M^{*}(K_{5}))$,
$(M(\qml{7}),M(K_{5}))$, or
$(M^{*}(\cml{10}),M^{*}(K_{5}))$.
Then $N$ is obtained from $M$ by trimming a bowtie ring,
deleting the central cocircuit from a good augmented $4$\dash wheel,
or a ladder-compression move.
\end{proposition}

\begin{proof}
\Cref{q8times} shows a labelling of some of the edges in $Q_{3}^{\times}$,
along with a good augmented $4$\dash wheel in $M^{*}(Q_{3}^{\times})$.
Deleting the central cocircuit of this augmented wheel produces $M^{*}(K_{5})$.
\Cref{fig19} shows the labelling of a bowtie ring in $M^{*}(Y_{9})$.
Trimming this ring produces $M^{*}(K_{5})$.
Similarly, by trimming the bowtie ring shown in \Cref{fig20},
we can obtain $M^{*}(K_{5})$ from $M^{*}(\cml{10})$.
Finally, it is clear that $M(\qml{n-2})$ is obtained from $M(\qml{n})$ by a
ladder-compression move, so in particular this applies to
$M(\qml{7})$ and $M(\qml{5})\cong M(K_{5})$.
\end{proof}

\begin{figure}[htb]
\centering
\includegraphics[scale=1.1]{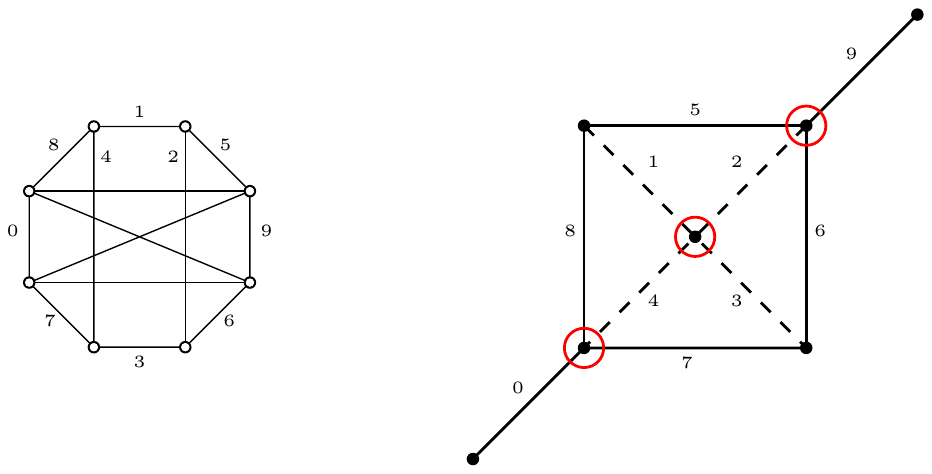}
\caption{$Q_3^\times$ and a good augmented $4$-wheel in $M^{*}(Q_{3}^{\times})$.}
\label{q8times}
\end{figure}

\begin{figure}[htb]
\centering
\includegraphics[scale=1.1]{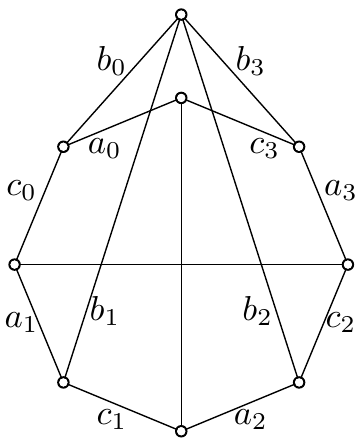}
\caption{A bowtie ring in $M^{*}(Y_{9})$.}
\label{fig19}
\end{figure}

\begin{figure}[htb]
\centering
\includegraphics[scale=1.1]{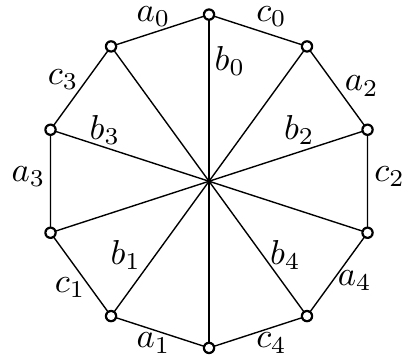}
\caption{A bowtie ring in $M^{*}(\cml{10})$.}
\label{fig20}
\end{figure}

Since \Cref{prop3} verifies \Cref{goodmove} for the pairs listed
in statement (4) of \Cref{repeat}, we now
move to non-graphic binary matroids.
We shall describe each of these matroids by giving a matrix
that is a reduced binary representation for it.
For example, \Cref{fig5} shows a matrix, $A$, which is a reduced
representation of $\Delta_{4}$.
\Cref{fig3} shows a geometric representation of $\Delta_{4}$.
Note that the element $9$ corresponds to $\gamma$,
so deleting $9$ produces a matroid isomorphic to
$M^{*}(\cml{6})\cong M^{*}(K_{3,3})$.

\begin{figure}[htb]
\centering
\includegraphics[scale=1.1]{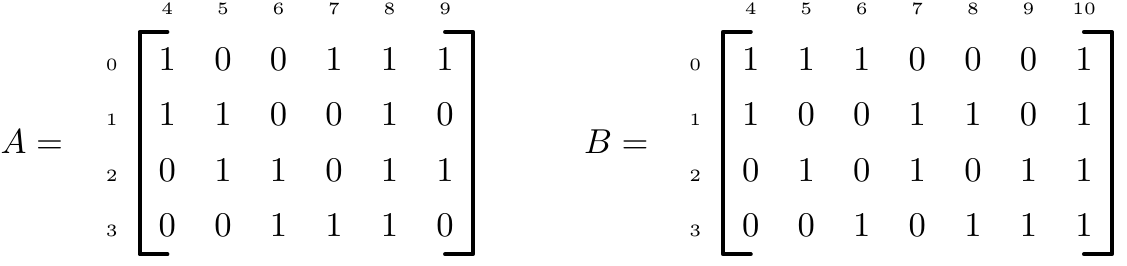}
\caption{Representations of $\Delta_{4}$ and $P$.}
\label{fig5}
\end{figure}

\begin{figure}[htb]
\centering
\includegraphics[scale=1.1]{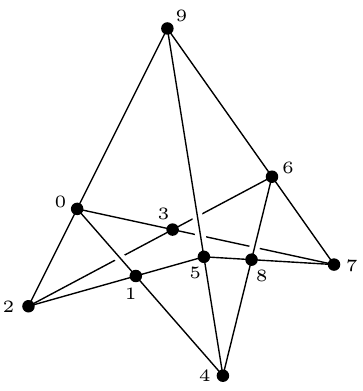}
\caption{A geometric representation of $\Delta_{4}$.}
\label{fig3}
\end{figure}

The matroids $A_{1}$, $A_{2}$, $A_{3}$, $A_{4}$, and $A_{5}$ have
as reduced representations the reduced matrices shown in \Cref{fig13}.
Thus each of $A_{1}$, $A_{2}$, $A_{3}$, $A_{4}$, and $A_{5}$
is a rank\dash $8$ binary matroid with $14$ elements, and each contains
a $4$\dash element independent set whose contraction produces a
minor isomorphic to $\Delta_{4}$.
The matroid $A_{6}$ is represented in \Cref{fig2}.
We can produce a $\Delta_{4}$\dash minor from $A_{6}$
by contracting a $3$\dash element independent set and deleting a single
element.

\begin{figure}[htb]
\centering
\includegraphics[scale=1.1]{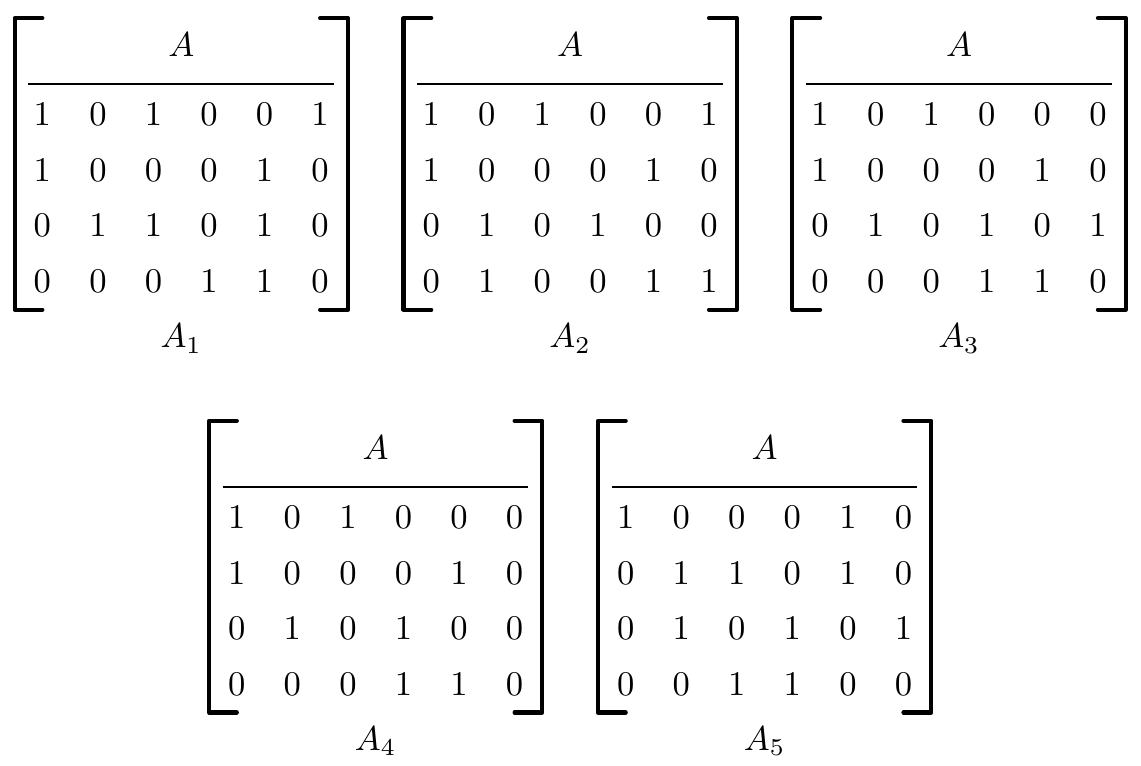}
\caption{Representations of $A_{1}$, $A_{2}$, $A_{3}$, $A_{4}$, and $A_{5}$.}
\label{fig13}
\end{figure}

\begin{figure}[htb]
\centering
\includegraphics[scale=1.1]{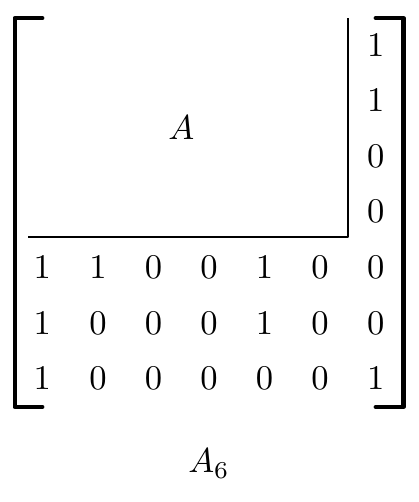}
\caption{A representation of $A_{6}$.}
\label{fig2}
\end{figure}

\begin{proposition}
\label{prop4}
Let $(M,N)$ be one of the pairs
$(A_{1}^{*},\Delta_{4}^{*})$,
$(A_{2}^{*},\Delta_{4}^{*})$,
$(A_{3}^{*},\Delta_{4}^{*})$,
$(A_{4}^{*},\Delta_{4}^{*})$,
$(A_{5}^{*},\Delta_{4}^{*})$, or
$(A_{6}^{*},\Delta_{4}^{*})$.
Then $N$ is obtained from $M$ by trimming a bowtie ring,
trimming an open rotor chain, or
deleting the central cocircuit from a good augmented $4$\dash wheel.
\end{proposition}

\begin{proof}
We will check that $\Delta_{4}^*$ is obtained from
each of $A_1^*,A_2^*,A_3^*$, and $A_5^*$ by trimming a bowtie
ring.
In \Cref{fig13}, assume that the matrices inherit the
labels on rows and columns from $A$, so that the first four
rows of any matrix are labelled $0$, $1$, $2$, $3$, the
columns are labelled $4$, $5$, $6$, $7$, $8$, $9$, and the
last four rows are labelled $10$, $11$, $12$, and $13$.
Now $A_{1}^{*}$ contains a bowtie ring, as in
\Cref{btringfig}, where $n=3$, and the labelling puts
\[(a_0,b_0,c_0,a_1,b_1,c_1,\dots ,a_3,b_3,c_3)=
(3,0,10,9,2,12,1,5,11,8,7,13).\]
Trimming this ring produces $\Delta_{4}^{*}$.
Similar statements apply to $A_2^*,A_3^*$, and $A_5^*$.
In those cases, the bowtie rings,
$(a_0,b_0,c_0,a_1,b_1,c_1,\dots ,a_3,b_3,c_3)$, are
\begin{linenomath}
\begin{multline*}
(4,8,11,5,7,12,0,3,10,2,6,13),\quad
(4,6,10,3,2,12,1,5,11,7,8,13),\\
\text{and}\quad (1,0,12,2,9,11,7,6,13,8,4,10).
\end{multline*}
\end{linenomath}

The matroid $A_4^*$ contains an open rotor chain,
as in \Cref{gmcdashed}, where $n=3$, and we label so that
\[(b_{0},c_0,a_1,b_1,c_1,a_{2},b_{2},c_{2},a_3,b_3,c_3)=
(2,10,3,6,13,4,8,11,7,5,12).\]
Trimming this rotor chain produces $\Delta_{4}^{*}$.

Finally, for $A_6$, we assume the matrix in \Cref{fig2} inherits
the labels from $A$, and we label the extra column $10$, and
the extra rows as $11,12$, and $13$.
Then $A_6^*$ contains an augmented $4$\dash wheel,
as in \Cref{a4w}, where we label so that
$(e,s,a_0,b_0,c_0,a_1,b_1,c_1,a_2,b_2)$
are replaced by $(1,0,13,10,4,11,12,5,8,7)$.
Now $A_6^*\ba 1$ is \ffsc, and
$A_{6}^{*}\ba 4,10,11,12\cong \Delta_{4}^{*}$, so the
proof of the \namecref{prop4} is complete.
\end{proof}

Before we continue, we recall some introductory material.
A simple rank\dash $r$ binary matroid, $M$, can be considered as a
subset, $E$, of points in the projective geometry $\mathrm{PG}(r-1,2)$.
The \emph{complement} of $M$ is the binary matroid corresponding to
the set of points of $\mathrm{PG}(r-1,2)$ not in $E$.
The complement of $M$ is well-defined by \cite[Proposition~10.1.7]{Oxl92},
meaning that it depends only on $M$, and not on the choice of $E$.
In particular, if two simple rank\dash $r$ binary matroids have isomorphic
complements, then they are themselves isomorphic.
The complement of $M^{*}(K_{3,3})$  in $\mathrm{PG}(3,2)$ is
$U_{2,3}\oplus U_{2,3}$, and the complement of $\Delta_{4}$ is
$U_{2,2}\oplus U_{2,3}$.
The complement of $M(K_{5})$ in $\mathrm{PG}(3,2)$ is 
$U_{4,5}$.
From this, it follows that $M(K_{5})$ has a unique simple
rank\dash $4$ binary extension on $11$ elements.
We denote this extension by $P$, so the complement
of $P$ is $U_{4,4}$.
The matrix $B$, shown in \Cref{fig5},
represents $P$ over $\mathrm{GF}(2)$.
Note that $P\ba 10$ is isomorphic to $M(K_{5})$, and that
$10$ is in triangles with $\{4,9\}$, $\{5,8\}$, and $\{6,7\}$,
where each of these pairs corresponds to a matching in $K_{5}$.
The matroids $B_{1}$, $B_{2}$, $B_{3}$, $B_{4}$, and $B_{5}$ are
represented by the matrices in \Cref{fig6}.  

\begin{figure}[htb]
\centering
\includegraphics[scale=1.1]{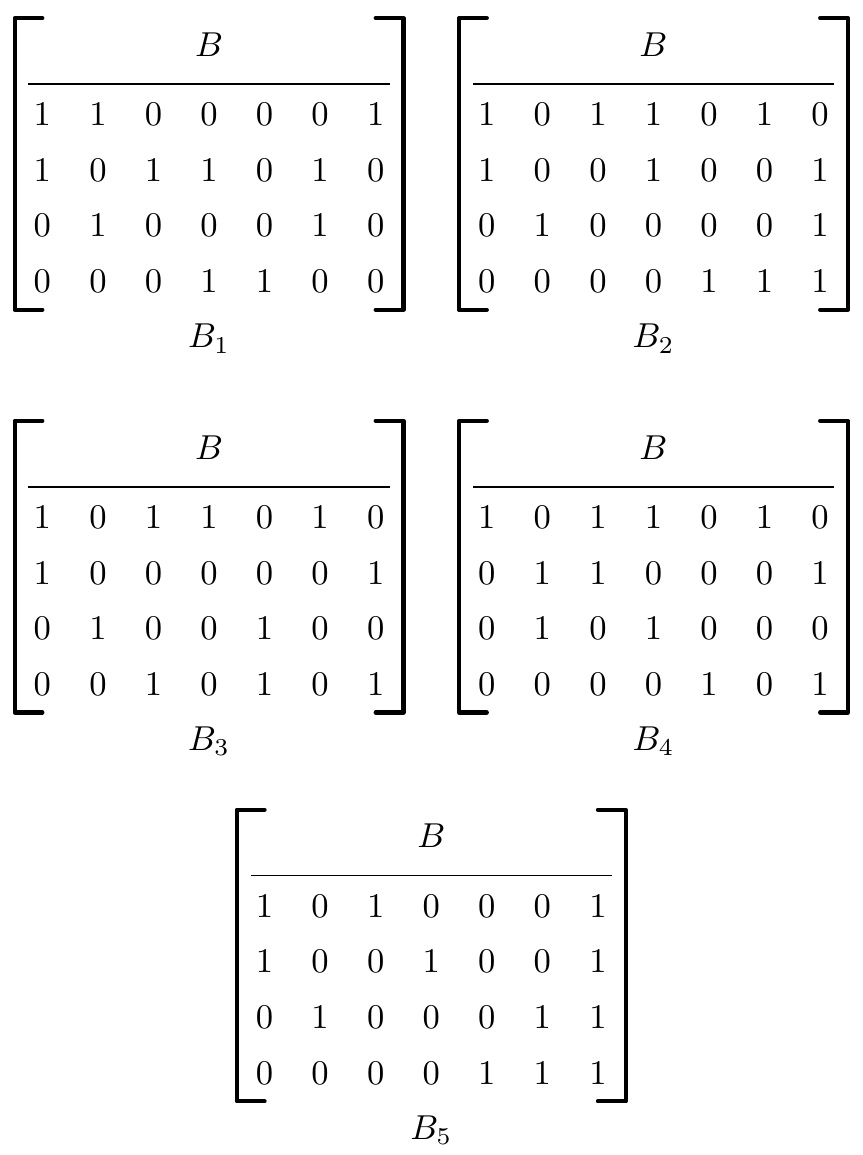}
\caption{Representations of $B_{1}$, $B_{2}$, $B_{3}$, $B_{4}$, and $B_{5}$.}
\label{fig6}
\end{figure}

\begin{proposition}
\label{prop5}
Let $(M,N)$ be one of the pairs
$(B_{1}^{*},P^{*})$,
$(B_{2}^{*},P^{*})$,
$(B_{3}^{*},P^{*})$,
$(B_{4}^{*},P^{*})$,
$(B_{5}^{*},P^{*})$.
Then $N$ is obtained from $M$ by trimming a bowtie ring.
\end{proposition}

\begin{proof}
We assume that each matrix, $B_i$, inherits the labels on $B$,
and that the extra rows are labelled $11$, $12$, $13$, and $14$.
In $B_{1}^{*}$, there is a bowtie ring, as in \Cref{btringfig}, with
$n=3$, where $(a_0,b_0,c_0,a_1,b_1,c_1,a_{2},b_{2},c_{2},a_3,b_3,c_3)$
is relabelled as $(1,3,12,0,6,11,5,9,13,7,8,14)$.
Similarly, for $B_2^*$, $B_3^*$, $B_4^*$, and $B_5^*$, the relevant
relabellings are $(1,8,12,10,5,13,2,0,11,6,3,14)$,
$(8,5,13,0,2,11,3,9,14,4,10,12)$,
$(10,8,14,3,1,11,0,4,12,7,5,13)$, and
$(8,1,12,7,2,13,5,0,11,6,3,14)$.
\end{proof}

Let $Q$ be the binary matroid represented by the matrix $C$, below.
Note that $Q$ is obtained by extending $\Delta_{4}$
by the element $10$ in such a way that $\{0,8,10\}$
is a triangle.
The complement of $Q$ in $\mathrm{PG}(3,2)$ is
$U_{1,1} \oplus U_{2,3}$.

\begin{center}
\includegraphics[scale=1.1]{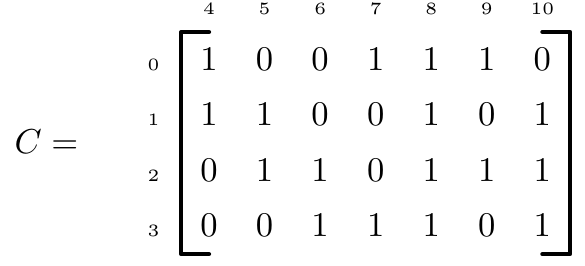}
\end{center}

The matroids $C_{1}$, $C_{2}$, $C_{3}$, and $C_{4}$ are represented by the
matrices in \Cref{fig7}.

\begin{figure}[htb]
\centering
\includegraphics[scale=1.1]{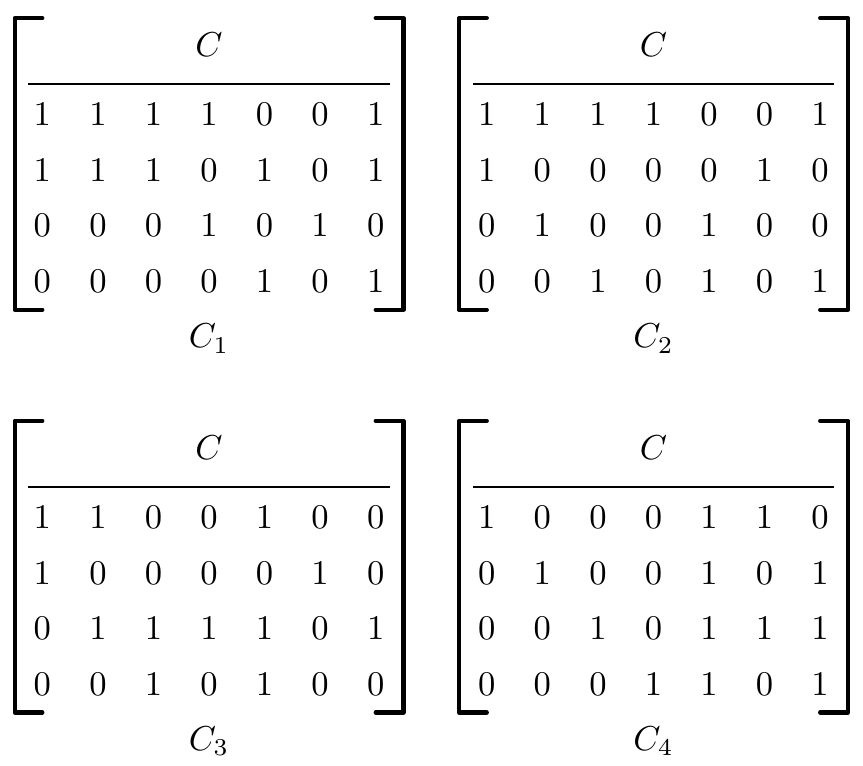}
\caption{Representations of $C_{1}$, $C_{2}$, $C_{3}$, and $C_{4}$.}
\label{fig7}
\end{figure}

\begin{proposition}
\label{prop6}
Let $(M,N)$ be one of the pairs
$(C_{1}^{*},Q^{*})$,
$(C_{2}^{*},Q^{*})$,
$(C_{3}^{*},Q^{*})$,
$(C_{4}^{*},Q^{*})$.
Then $N$ is obtained from $M$ by trimming a bowtie ring.
\end{proposition}

\begin{proof}
We assume that each matrix $C_i$ inherits the row and column labels
from $C$, and the extra rows are labelled $11$, $12$, $13$, and $14$.
We relabel the elements
$(a_0,b_0,c_0,a_1,b_1,c_1,a_{2},b_{2},c_{2},a_3,b_3,c_3)$ in \Cref{btringfig}
$(1,6,12,7,9,13,2,0,11,8,10,14)$ for $C_{1}^{*}$,
$(4,9,12,2,0,11,3,7,14,8,5,13)$ for $C_{2}^{*}$,
$(9,4,12,8,6,14,1,10,11,3,5,13)$ for $C_{3}^{*}$,
and $(7,0,11,4,1,12,5,2,13,6,3,14)$ for $C_{4}^{*}$.
\end{proof}

\Cref{prop4,prop5,prop6}
verify \Cref{goodmove}
for the pairs listed in statements
(5), (6), and (7) in \Cref{repeat}.
There are two matrices in \Cref{fig8}.
The matrix $D$ represents the binary matroid $R$.
Note that $R$ is obtained from $M(K_{5})$ by coextending by
the element $10$ so that $10$ is in a triad with two elements
that correspond to a $2$\dash edge matching in $K_{5}$.
Therefore $R$ is isomorphic to the matroid
obtained from $P$ by performing a $\Delta\text{-}Y$\dash operation
on the triangle $\{4,9,10\}$.

\begin{figure}[htb]
\centering
\includegraphics[scale=1.1]{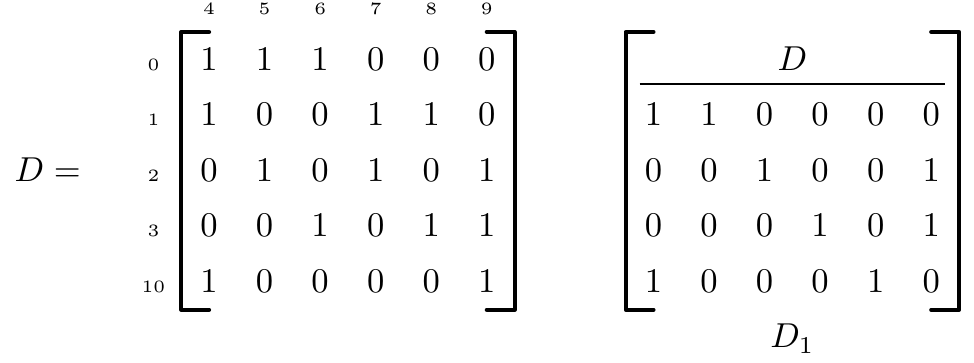}
\caption{Representations of $R$ and $D_{1}$.}
\label{fig8}
\end{figure}

\begin{proposition}
\label{prop7}
$R^{*}$ can be obtained from $D_{1}^{*}$ by trimming
a bowtie ring.
\end{proposition}

\begin{proof}
Label the extra rows in $D_1$ that are not in $D$ as
$11$, $12$, $13$, and $14$.  
Then $(8,3,12,6,0,11,5,2,13,7,1,14)$ is the appropriate
bowtie ring.
\end{proof}

The matroid $S$ is represented by the matrix $E$, and $E_{1}$
is represented by the matrix shown in \Cref{fig10}.
We can obtain $S$ from $\Delta_{4}$ by coextending by
the element $10$ so that it is in a triad with $0$ and $8$.
Thus $S$ can also be obtained from $Q$ by 
a $\Delta\text{-}Y$\dash operation.

\begin{figure}[htb]
\centering
\includegraphics[scale=1.1]{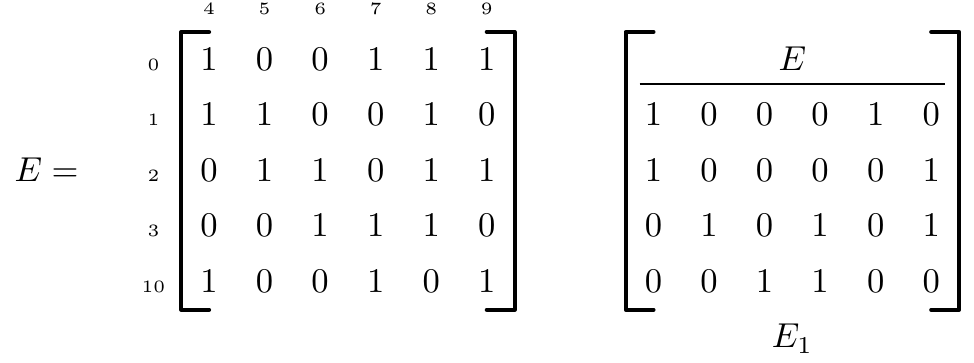}
\caption{Representations of $S$ and $E_{1}$.}
\label{fig10}
\end{figure}

\begin{proposition}
\label{prop8}
$S^{*}$ can be obtained from $E_{1}^{*}$ by trimming
a bowtie ring.
\end{proposition}

\begin{proof}
Label the extra rows in $E_1$ that are not in $E$
as $11$, $12$, $13$, and $14$.
Then $(1,5,11,4,9,12,7,6,14,3,2,13)$ is the appropriate
bowtie ring.
\end{proof}

Recall that the M\"{o}bius matroids are defined in \Cref{sect1}.

\begin{proposition}
\label{prop1}
When $r\geq 6$ is an even integer, the matroid $\Upsilon_{r}^{*}$
can be obtained from $\Upsilon_{r+2}^{*}$ by a ladder-compression move.
\end{proposition}

\begin{proof}
Recall that $\Upsilon_{r+2}^{*}=M_{r+1}$ and $\Upsilon_{r}^{*}=M_{r-1}$,
where $M_{k}$ is an extension of the rank\dash $k$ wheel by the element
$\gamma$.
Assume that the spokes of $M({\mathcal W}_{r+1})$, in cyclic order, are
$x_0,x_1,\ldots,x_{r}$ and that $\{x_i,y_i,x_{i+1}\}$ is a triangle of
$M({\mathcal W}_{r+1})$ for $i=0,1,\ldots, r$.
(We interpret subscripts modulo $r+1$.)
Then, for $i=0,1,\ldots, r$, the set
$\{y_{i},x_{i+1},y_{i+1},\gamma\}$ is a cocircuit of $M_{r+1}$.
We obtain $M_{r-1}$ from $M_{r+1}$
by contracting $y_{r-1}$ and $y_{r}$,
and deleting $x_{r-1}$ and $x_{0}$, and then relabelling $x_{r}$
as $x_{0}$.
To see this, observe that $M_{r+1}$ has
$\{x_{0},\ldots, x_{r},\gamma\}$
and $\{x_{r-1},x_{r},y_{r-1}\}$ as circuits, so their symmetric difference,
$C=\{x_{0},\ldots, x_{r-2}, y_{r-1},\gamma\}$, is a disjoint union of circuits.
Orthogonality with the cocircuits containing $\gamma$ implies that $C$ is a
circuit of $M_{r+1}$.
Next we note that $\{x_{r-1},x_{r},y_{r-2},y_{r}\}$ is the
symmetric difference of $\{y_{r-2},x_{r-1},y_{r-1},\gamma\}$
and $\{y_{r-1},x_{r},y_{r},\gamma\}$, and is therefore
a disjoint union of cocircuits.
This implies that $y_{r}$ is not in the closure of $C$ in $M_{r+1}$.
Therefore $C-y_{r-1}=\{x_{0},\ldots, x_{r-2},\gamma\}$ is a spanning
circuit of $M_{r+1}/y_{r-1},y_{r}\ba x_{r-1},x_{0}$, and it follows
easily that this matroid is $M_{r-1}$, up to relabelling.

Now we need only show that this operation is a ladder-compression move. 
We note that $M_{r+1}$ contains a ladder segment, as depicted in
\Cref{laddery}, where the labels
$a_0$, $b_0$, $c_0$, $d_0$, $a_1$, $b_1$, $c_1$,
$d_1$, $a_2$, $b_2$, $c_2$, and $d_2$
are replaced by
$x_{r-4}$, $y_{r-4}$, $x_{r-3}$, $y_{r-3}$, $x_{r-2}$,  $y_{r-2}$,
$x_{r-1}$, $y_{r-1}$, $x_{r}$, $y_{r}$, $x_0$, and $y_0$, respectively.
Because $r\geq 6$, these elements are all distinct.
\end{proof}

\Cref{prop1} now implies that $\Upsilon_{6}^{*}$ can be obtained
from $\Upsilon_{8}^{*}$ by a ladder-compression move.
Thus we have completed the proof of \Cref{goodmovefirst}.

\begin{proof}[Proof of \Cref{cor2}.]
If $(M,N)$ is $(M(\qml{7}),M(K_{4}))$, then we can set
$M_{0}$ to be $M(\qml{5})\cong M(K_{5})$, and $M_{0}$ can be
obtained from $M$ by a ladder-compression move.
If $(M,N)$ is $(\Upsilon_{8},F_{7})$ or $(\Upsilon_{8}^{*},F_{7})$,
then we can set $M_{0}$ to be $\Upsilon_{6}$ or $\Upsilon_{6}^{*}$,
respectively.
In either case, by \Cref{prop1}, we can use a ladder-compression
move to obtain $M_{0}^{*}$ from $M^{*}$ (in the first case), or
$M_{0}$ from $M$ (in the second).
\end{proof}

\section{Proof of the main results}

We prove
\Cref{maintheorem}.
Assume that $(M,N)$ is a fascinating pair that
contradicts the statement of the \namecref{maintheorem}.

\addtocounter{theorem}{1}

\begin{sublemma}
\label{sub1}
$|E(N)|\in\{10,11\}$.
\end{sublemma}

Certainly $|E(N)|\leq 11$, since $|E(M)|\leq 15$, and
$(M,N)$ is a fascinating pair, so $|E(M)|-|E(N)|>3$.
Assume that $|E(N)|<10$.
First consider the case that $|E(N)|=6$, so that
$N$ is isomorphic to $M(K_{4})$.
If $M$ has a proper minor, $M'$, such that
$|E(M)|-|E(M')|\leq 3$, and $M'$ is \ifc,
then $M'$ has an $M(K_{4})$\dash minor
\cite[Corollary~12.2.13]{Oxl11},
and hence $(M,N)$ is not a fascinating pair.
Therefore $M$ has no such minor, so we
can apply our chain theorem \cite[Theorem~1.3]{CMO11}.
Since $|E(M)|\leq 15$, it follows from that theorem
that $M$ is the cycle matroid of a planar or M\"{o}bius quartic
ladder, or the dual of such a matroid.
The only planar quartic ladder with fewer than $16$
edges is the octahedron, $O$, which is the dual graph of
$Q_{3}$, the cube.
The only M\"obius quartic ladders with fewer than $16$ edges have
$14$ or $10$ edges.
The former has the latter as a minor, and the latter is 
isomorphic to $K_{5}$.
From this we deduce that, up to duality,
$(M,N)$ is $(M(Q_{3}),M(K_{4}))$ or $(M(K_{5}),M(K_{4}))$,
and that therefore $(M,N)$ is not a counterexample
after all.
Hence $6<|E(N)|<10$.
The only \ifc\ binary matroids satisfying this constraint are
$F_{7}$, $M(K_{3,3})$, and their duals.
(This fact is \cite[Lemma~2.1]{GZ06}, and will also be confirmed
by the subsequent exhaustive search.)
Thus we can assume $N$ is $F_{7}$ or $M(K_{3,3})$.

From this point, we use almost exactly the same arguments as in
\cite[Lemma~2.3]{CMO13}.
Assume $N$ is $F_{7}$, so $|E(M)|\geq 11$.
We can use \cite[Corollary~1.2]{Zho04} to deduce that
$M$ is isomorphic to
$T_{12}\ba e\cong \Upsilon_{6}$ or $T_{12}/e\cong \Upsilon_{6}^{*}$,
so $(M,N)$ fails to contradict the theorem.
Therefore we assume $N$ is $M(K_{3,3})$,
and hence $|E(M)|\geq 13$.
Now we can use \cite[Lemma~2.3]{GZ06}.
This lemma defines five graphs, but only four of them
have at least $13$ edges.
Therefore we can deduce that $M$ is isomorphic to one of
the graphic matroids $M(H_{1})$, $M(H_{2})$, $M(H_{3})$,
or $M(\qml{7})$.
Again this is a contradiction, as it implies that
$(M,N)$ is not a counterexample, so the proof of
\Cref{sub1} is complete.

At this point, it is appropriate to verify that the pairs
mentioned in the proof of \Cref{sub1} are indeed fascinating.
We do this, and the rest of the search, using the matroid
capabilities of sage (Version 6.10).
First we want to allow access to certain special functions of the
sage matroids package.

{\small\begin{verbatim}
from sage.matroids.advanced import *
\end{verbatim}}

We will require a test for internal $4$\dash connectivity.

{\small\begin{verbatim}
def IsIFC(M):
    if len(M)<=7:
        return True
    elif len(M)==8 or len(M)==9:
        return all( (M.rank(X)+
             M.rank(M.groundset().symmetric_difference(X))-
             M.rank() > 2) for X in Subsets(M.groundset(),4) )
    elif len(M)==10 or len(M)==11:
        return all( (M.rank(X)+
            M.rank(M.groundset().symmetric_difference(X))-
            M.rank() > 2) for X in Subsets(M.groundset(),4) )
            and all( (M.rank(X)+
            M.rank(M.groundset().symmetric_difference(X))-
            M.rank() > 2) for X in Subsets(M.groundset(),5) )
    elif len(M)==12 or len(M)==13:
        return all( (M.rank(X)+
            M.rank(M.groundset().symmetric_difference(X))-
            M.rank() > 2) for X in Subsets(M.groundset(),4) )
            and all( (M.rank(X)+
            M.rank(M.groundset().symmetric_difference(X))-
            M.rank() > 2) for X in Subsets(M.groundset(),5) )
            and all( (M.rank(X)+
            M.rank(M.groundset().symmetric_difference(X))-
            M.rank() > 2) for X in Subsets(M.groundset(),6) )
    elif len(M)==14 or len(M)==15:
        return all( (M.rank(X)+
            M.rank(M.groundset().symmetric_difference(X))-
            M.rank()>2) for X in Subsets(M.groundset(),4) )
            and all( (M.rank(X)+
            M.rank(M.groundset().symmetric_difference(X))-
            M.rank() > 2) for X in Subsets(M.groundset(),5) )
            and all( (M.rank(X)+
            M.rank(M.groundset().symmetric_difference(X))-
            M.rank() > 2) for X in Subsets(M.groundset(),6) )
            and all( (M.rank(X)+
            M.rank(M.groundset().symmetric_difference(X))-
            M.rank() > 2) for X in Subsets(M.groundset(),7) )
\end{verbatim}}

This command works for $3$\dash connected matroids
with a ground set of size $n$, where $0\leq n\leq 15$.
For each such $n$, the command considers each subset,
$X$, of size between $4$ and $\lfloor n/2 \rfloor$, and checks
that $r(X)+r(E(M)-X)-r(M)>2$.
If this is the case, it returns \verb!True!, and otherwise it
returns \verb!False!.

Next we define a function that will test whether
a pair $(M,N)$ is fascinating.
In the following code, and elsewhere, note that the
command \verb!range(n)! produces the list
$[0,1,2,\ldots,n-1]$, and
\verb!range(m,n)! produces
$[m,m+1,\ldots,n-1]$.

{\small\begin{verbatim}
def Fascinating(M,N):
    rankgap=M.rank()-N.rank()
    sizegap=len(M)-len(N)
    if sizegap>3 and M.has_minor(N):
        Between=False
        for r in range(rankgap+1):
            if Between:
                break
            for F in M.flats(r):
                if Between:
                    break
                if len(F)<sizegap and M.contract(F).has_minor(N):
                    if r==0:
                        Lower=1
                    else:
                        Lower=0
                    DeleteSet=M.groundset().difference(F)
                    for i in range(Lower,sizegap-len(F)):
                        if Between:
                            break
                        for D in Subsets(DeleteSet,i):
                            Test=M.contract(F).delete(D)
                            if Test.has_minor(N):
                                if Test.is_3connected()
                                    and IsIFC(Test):
                                       Between=True
                                       break
        return not Between
    else:
        return False
\end{verbatim}}

First the function tests that $M$ has an $N$\dash minor and
$|E(M)|-|E(N)|>3$.
If this is not the case, it returns \verb!False!.
Otherwise, it considers all flats, $F$, of $M$
such that $0\leq r(F) \leq r(M)-r(N)$.
If $M/F$ has a proper $N$\dash minor,
then it considers subsets, $D$, of
$E(M/F)$.
If $F$ is the rank\dash $0$ flat (which we assume to be
empty), then $D$ is constrained to contain
at least one element.
In any case, $D$ is constrained so that
$|D|<|E(M)|-|E(N)|-|F|$.
Thus $D$ ranges over all subsets such that
$|E(N)|<|E(M/F\ba D)|<|E(M)|$.
If $M/F\ba D$ is \ifc\
and has an $N$\dash minor, then the Boolean
value \verb!Between! is set to be \verb!True!.
At any time, if \verb!Between! is found to be \verb!True!,
then the function breaks out of the loop.
Finally, it returns the negation of \verb!Between!.

Now we can test the pairs that have arisen in the proof up to this
point.

{\small\begin{verbatim}
K4=matroids.CompleteGraphic(4)
K5=matroids.CompleteGraphic(5)
Q3=Matroid(graph=graphs.CubeGraph(3))
F7=matroids.named_matroids.Fano()
Upsilon6=matroids.named_matroids.T12().delete('e')
K33=Matroid(graph=graphs.CompleteBipartiteGraph(3,3))
H1=Matroid(graph=Graph({0:[1,2,4,5],1:[2,3,4,5],2:[3,4],
   3:[4,5],4:[5]}))
H2=Matroid(graph=Graph({0:[1,3,5],1:[2,4,6],2:[3,5,6],
   3:[4,6],4:[5,6]}))
H3=Matroid(graph=Graph({0:[1,3,7],1:[2,6],2:[3,5,7],
   3:[4],4:[5,7],5:[6],6:[7]}))
QML7=Matroid(graph=graphs.CirculantGraph(7,[1,3,4]))
\end{verbatim}}

{\small\begin{verbatim}
print Fascinating(K5,K4)
print Fascinating(Q3,K4)
print Fascinating(Upsilon6,F7)
print Fascinating(Upsilon6.dual(),F7)
print Fascinating(H1,K33)
print Fascinating(H2,K33)
print Fascinating(H3,K33)
print Fascinating(QML7,K33)
    True
    True
    True
    True
    True
    True
    True
    True
\end{verbatim}}

By duality, we may assume that $r(M)\leq r^{*}(M)$.
As $|E(M)|\leq 15$, the next result is a consequence.

\begin{sublemma}
$r(N)\leq r(M)\leq 7$.
\end{sublemma}

We create an object that will contain the catalogue of all $3$\dash connected
binary matroids with ground sets of cardinality between $6$ and $15$
and rank at most $7$.
This object is a library, containing 10 lists, each indexed by an integer between
$6$ and $15$.
Each list itself contains eight lists, indexed by integers between $0$ and $7$.
Thus, if $6\leq n\leq 15$, and $0\leq r\leq 7$, then
\verb!Catalogue[n][r]! is the list indexed by $r$, contained in the list indexed by $n$;
that is, it is the list of all $3$\dash connected binary matroids
with rank $r$ and a ground set of size $n$.
We initialise by creating empty lists.

{\small\begin{verbatim}
Catalogue={}
for i in range(6,16):
    Catalogue[i]=[[] for j in range(0,8)]
\end{verbatim}}

Every $3$\dash connected binary matroid with at least
$6$ elements contains an $M(K_{4})$\dash minor
\cite[Corollary~12.2.13]{Oxl11}.
We are going to populate our catalogue by starting
with this matroid, and enlarging the catalogue
through single-element extensions
and coextensions.
When we extend, we ensure we
produce no coloops, no loops, and no parallel pairs.
Dually, when we coextend, we create no
loops, coloops, or series pairs.
Thus we only ever create
$3$\dash connected matroids
\cite[Proposition~8.1.10]{Oxl92}.
Every $3$\dash connected binary matroid can be
constructed in this way, with the exception of
wheels \cite[Theorem~8.8.4]{Oxl11}, so we manually input
the wheels of rank $3$, $4$, $5$, $6$, and $7$.
In this way, we guarantee that our catalogue will
contain every $3$\dash connected binary matroid
with suitable size and rank.

{\small\begin{verbatim}
Wheel3=Matroid(reduced_matrix=Matrix(GF(2),[[1,0,1],[1,1,0],
    [0,1,1]]))
Wheel4=Matroid(reduced_matrix=Matrix(GF(2),[[1,0,0,1],
    [1,1,0,0],[0,1,1,0],[0,0,1,1]]))
Wheel5=Matroid(reduced_matrix=Matrix(GF(2),[[1,0,0,0,1],
    [1,1,0,0,0],[0,1,1,0,0],[0,0,1,1,0],[0,0,0,1,1]]))
Wheel6=Matroid(reduced_matrix=Matrix(GF(2),[[1,0,0,0,0,1],
    [1,1,0,0,0,0],[0,1,1,0,0,0],[0,0,1,1,0,0],[0,0,0,1,1,0],
    [0,0,0,0,1,1]]))
Wheel7=Matroid(reduced_matrix=Matrix(GF(2),[[1,0,0,0,0,0,1],
    [1,1,0,0,0,0,0],[0,1,1,0,0,0,0],[0,0,1,1,0,0,0],
    [0,0,0,1,1,0,0],[0,0,0,0,1,1,0],[0,0,0,0,0,1,1]]))
\end{verbatim}}

Now we seed our catalogue by adding the wheels.

{\small\begin{verbatim}
Catalogue[6][3].append(Wheel3)
Catalogue[8][4].append(Wheel4)
Catalogue[10][5].append(Wheel5)
Catalogue[12][6].append(Wheel6)
Catalogue[14][7].append(Wheel7)
\end{verbatim}}

At any time, we can save our current catalogue
with a command of the following type.

{\small\begin{verbatim}
save (Catalogue, "/SampleFolder/catalogue.sobj")
\end{verbatim}}

We are then able to recover our saved work with the
following command.

{\small\begin{verbatim}
Catalogue = load("/SampleFolder/catalogue.sobj")
\end{verbatim}}

Now we define the command \verb!Populate!,
which we will use to fill in the entries in our
catalogue.
On input $n$, the command fills in all
entries of the catalogue corresponding to
matroids with ground sets of size $n$.
It does this by letting the rank, $r$, range from $0$ to $7$,
and considering all matroids of rank $r$
in the catalogue of matroids with ground sets of size $n-1$.
For each such matroid, $N$, it generates
the list of non-isomorphic simple single-element
extensions of $N$, using the built-in command
\verb!get_nonisomorphic_matroids!.
It then adds these extensions to the catalogue of
rank\dash $r$, size\dash $n$ matroids, as long as
they are not isomorphic to matroids already appearing there.
If $r$ is greater than zero, it then performs the same actions
using cosimple single-element coextensions.
Finally, it prints the number of matroids it
has generated.

{\small\begin{verbatim}
def Populate(n):
    for r in range(8):
        for N in Catalogue[n-1][r]:
            List=N.linear_extensions(simple=True,element=len(N))
            List=get_nonisomorphic_matroids(List)
            for M in List:
                if not any(L.is_isomorphic(M) for
                    L in Catalogue[n][r]):
                    Catalogue[n][r].append(M)
        if r>0:
            for N in Catalogue[n-1][r-1]:
                List=N.linear_coextensions(cosimple=True,
                        element=len(N))
                List=get_nonisomorphic_matroids(List)
                for M in List:
                    if not any(L.is_isomorphic(M) for
                        L in Catalogue[n][r]):
                        Catalogue[n][r].append(M)
    print [len(Catalogue[n][r]) for r in range(8)]
\end{verbatim}}

Generating the catalogue of $3$\dash connected binary matroids up to
size $13$ takes only a few minutes.

{\small\begin{verbatim}
%time
Populate(7)
    [0, 0, 0, 1, 1, 0, 0, 0]
    CPU time: 0.00 s,  Wall time: 0.00 s
\end{verbatim}}

{\small\begin{verbatim}
%time
Populate(8)
    [0, 0, 0, 0, 3, 0, 0, 0]
    CPU time: 0.01 s,  Wall time: 0.01 s
\end{verbatim}}

{\small\begin{verbatim}
%time
Populate(9)
    [0, 0, 0, 0, 4, 4, 0, 0]
    CPU time: 0.03 s,  Wall time: 0.03 s
\end{verbatim}}

{\small\begin{verbatim}
%time
Populate(10)
    [0, 0, 0, 0, 4, 16, 4, 0]
    CPU time: 0.20 s,  Wall time: 0.20 s
\end{verbatim}}

{\small\begin{verbatim}
%time
Populate(11)
    [0, 0, 0, 0, 3, 37, 37, 3]
    CPU time: 1.97 s,  Wall time: 1.97 s
\end{verbatim}}

{\small\begin{verbatim}
%time
Populate(12)
    [0, 0, 0, 0, 2, 68, 230, 68]
    CPU time: 18.58 s,  Wall time: 18.58 s
\end{verbatim}}

{\small\begin{verbatim}
%time
Populate(13)
    [0, 0, 0, 0, 1, 98, 983, 983]
    CPU time: 218.23 s,  Wall time: 218.26 s
\end{verbatim}}

Generating matroids with $14$ and $15$ elements takes considerably
more time.

{\small\begin{verbatim}
%time
Populate(14)
    [0, 0, 0, 0, 1, 121, 3360, 10035]
    CPU time: 3243.90 s,  Wall time: 3243.96 s
\end{verbatim}}

{\small\begin{verbatim}
%time
Populate(15)
    [0, 0, 0, 0, 1, 140, 10012, 81218]
    CPU time: 83469.47 s,  Wall time: 83468.83 s
\end{verbatim}}

Now we work through our catalogue of all $3$\dash connected
binary matroids, and pick out those that are
\ifc.
As before, we create an object that will be our
catalogue of \ifc\ matroids.

{\small\begin{verbatim}
IFCCatalogue={}
for i in range(6,16):
    IFCCatalogue[i]=[[] for j in range(0,8)]
\end{verbatim}}

We construct a command that will populate our
catalogue with \ifc\ matroids.

{\small\begin{verbatim}
def PopulateIFC(n):
    for r in range(8):
        for M in Catalogue[n][r]:
            if IsIFC(M):
                IFCCatalogue[n][r].append(M)
    print [len(IFCCatalogue[n][r]) for r in range(8)]
\end{verbatim}}

These are the results when we apply the command to
fill our catalogue.

{\small\begin{verbatim}
%time
PopulateIFC(6)
    [0, 0, 0, 1, 0, 0, 0, 0]
    CPU time: 0.00 s,  Wall time: 0.00 s
\end{verbatim}}

{\small\begin{verbatim}
%time
PopulateIFC(7)
    [0, 0, 0, 1, 1, 0, 0, 0]
    CPU time: 0.00 s,  Wall time: 0.00 s
\end{verbatim}}

{\small\begin{verbatim}
%time
PopulateIFC(8)
    [0, 0, 0, 0, 0, 0, 0, 0]
    CPU time: 0.00 s,  Wall time: 0.00 s
\end{verbatim}}

{\small\begin{verbatim}
%time
PopulateIFC(9)
    [0, 0, 0, 0, 1, 1, 0, 0]
    CPU time: 0.02 s,  Wall time: 0.02 s
\end{verbatim}}

Since $F_{7}$ and $M(K_{3,3})$ are \ifc,
from these commands it follows that, as stated earlier, the only
\ifc\ binary matroids
with seven, eight, or nine elements are $F_{7}$, $M(K_{3,3})$,
and their duals.

{\small\begin{verbatim}
%time
PopulateIFC(10)
    [0, 0, 0, 0, 2, 2, 2, 0]
    CPU time: 0.14 s,  Wall time: 0.14 s
\end{verbatim}}

{\small\begin{verbatim}
%time
PopulateIFC(11)
    [0, 0, 0, 0, 2, 7, 7, 2]
    CPU time: 0.76 s,  Wall time: 0.78 s
\end{verbatim}}

{\small\begin{verbatim}
%time
PopulateIFC(12)
    [0, 0, 0, 0, 2, 24, 46, 24]
    CPU time: 9.60 s,  Wall time: 9.60 s
\end{verbatim}}

{\small\begin{verbatim}
%time
PopulateIFC(13)
    [0, 0, 0, 0, 1, 52, 272, 272]
    CPU time: 94.84 s,  Wall time: 94.87 s
\end{verbatim}}

{\small\begin{verbatim}
%time
PopulateIFC(14)
    [0, 0, 0, 0, 1, 84, 1389, 3385]
    CPU time: 1829.54 s,  Wall time: 1829.78 s
\end{verbatim}}

{\small\begin{verbatim}
%time
PopulateIFC(15)
    [0, 0, 0, 0, 1, 116, 5816, 36962]
    CPU time: 26298.59 s,  Wall time: 26297.77 s
\end{verbatim}}

These computations show that there are
exactly $24$ \ifc\ binary matroids
with ground sets of cardinality $10$ or $11$.
We put these matroids into a list of possible ``target'' matroids:

{\small\begin{verbatim}
Targets=[]
Targets.extend(IFCCatalogue[10][4])
Targets.extend(IFCCatalogue[10][5])
Targets.extend(IFCCatalogue[10][6])
Targets.extend(IFCCatalogue[11][4])
Targets.extend(IFCCatalogue[11][5])
Targets.extend(IFCCatalogue[11][6])
Targets.extend(IFCCatalogue[11][7])
\end{verbatim}}

We wish to process each of the \ifc\
matroids in our catalogue with a ground set of cardinality
$11$, $12$, $13$, or $14$, and record which of the
$24$ ``target'' matroids it has as a proper minor.
We start by creating a new list of lists of lists.

{\small\begin{verbatim}
TargetMinors={}
for n in range(11,15):
    TargetMinors[n]={}
    for r in range(8):
        TargetMinors[n][r]={}
        for i in range(len(IFCCatalogue[n][r])):
            TargetMinors[n][r][i]=[0 for j in range(24)]
\end{verbatim}}

Assume that $n$ is in $\{11,12,13,14\}$, $r$ is in $\{0,\ldots,7\}$,
and $i$ indexes an \ifc\ matroid
in the list \verb!IFCCatalogue[n][r]!.
The object \verb!TargetMinors[n][r][i]! is a list with
$24$ entries, each equal to $0$ or $1$.
The entry \verb!TargetMinors[n][r][i][j]! will be equal to $1$
if and only if matroid number $i$ in the list
\verb!IFCCatalogue[n][r]! contains a proper minor
isomorphic to matroid number $j$ in the list
\verb!Targets!.

{\small\begin{verbatim}
%time
for n in range(11,15):
    for r in range(8):
        for i in range(len(IFCCatalogue[n][r])):
            M=IFCCatalogue[n][r][i]
            for j in range(24):
                N=Targets[j]
                if len(N)<len(M) and M.has_minor(N):
                    TargetMinors[n][r][i][j]=1
    CPU time: 1797.00 s,  Wall time: 1797.06 s
\end{verbatim}}

The heart of the computer search is contained in the
following code:

{\small\begin{verbatim}
%time
for r in range(8):
    for i in range(len(IFCCatalogue[15][r])):
        M=IFCCatalogue[15][r][i]
        Possibles=set(range(24))
        for k in range(24):
            if not M.has_minor(Targets[k]):
                Possibles.discard(k)
        for size in range(14,10,-1):
            if len(Possibles)==0:
                break
            for rank in range(r,max(r-(15-size),0)-1,-1):
                if len(Possibles)==0:
                    break
                for j in range(len(IFCCatalogue[size][rank])):
                    Test=False
                    for k in Possibles:
                        if TargetMinors[size][rank][j][k]==1:
                            Test=True
                            break
                    if Test:
                        if M.has_minor(IFCCatalogue[size][rank][j]):
                            for k in range(24):
                                if TargetMinors[size][rank][j][k]==1:
                                    Possibles.discard(k)
                            if len(Possibles)==0:
                                break
        for k in Possibles:
            print (15,r,i,k)
\end{verbatim}}

This code lets the variable $r$ range between $0$ and $7$.
For each $r$, the variable $i$ indexes a matroid,
$M$, contained in the list \verb!IFCCatalogue[15][r]!, and $M$
ranges over all matroids in this list.
We are trying to determine if there is an
fascinating pair, $(M,N)$, that acts as a counterexample
to the theorem.
Recall that \Cref{sub1} implies that $|E(N)|$ is $10$ or $11$,
and \verb!Targets! contains a list of the $24$ matroids
with $10$ or $11$ elements that could possibly be $N$.
When processing $M$, we set \verb!Possibles! to
initially be $\{0,1,\ldots, 23\}$.
This set will record the indices of matroids in
\verb!Target! that can possibly be $N$.
If $M$ does not contain a minor isomorphic
to \verb!Targets[k]!, then we discard $k$ from
\verb!Possibles!.
Indeed, throughout the process, we seek to
discard indices from \verb!Possibles!.
If at any time, there are no indices left
in \verb!Possibles!, then we know that $M$ cannot
be in the fascinating pair we seek, so we move to the next
matroid.
Assuming that \verb!Possibles! is not empty, we
let \texttt{size} range between $14$ and $11$ (starting at
$14$), and we let
\verb!rank! range between $r$ and $r-(15-\texttt{size})$.
We seek \ifc\ matroids with
the parameters \texttt{size} and \verb!rank! that are proper
minors of $M$.
(This is why we do not allow \verb!rank! to be less than
$r-(15-\texttt{size})$, since no minor of $M$ with ground set
of cardinality \verb!size! can have rank lower than this.)
Now we let $j$ index a matroid in the list
\verb!IFCCatalogue[size][rank]!.
Let the indexed matroid be $M'$.
If $M'$ has no minor isomorphic to \verb!Targets[k]!, for any
$k$ in \verb!Possibles!, then we will not be able to use
$M'$ to discard any more indices in \verb!Possibles!.
So if this is the case, we do not consider $M'$ any further.
On the other hand, as soon as we find a $k$
in \verb!Possibles! such that $M'$ has a minor
isomorphic to \verb!Targets[k]!, we move to the next stage of
the process.
In this stage, we first of all test that $M$ has an
$M'$\dash minor.
In the case that it does,
we let $k$ range through all of the indices
in $\{0,1,\ldots, 23\}$.
If \verb!TargetMinors[size][rank][j][k]! is $1$,
indicating that $M'$ has a minor isomorphic to
\verb!Targets[k]!, then we can discard $k$
from \verb!Possibles!, since the matroid
$M'$ means that no fascinating pair can contain $M$ and
\verb!Targets[k]!.
We continue this process until either
\verb!Possibles! is empty, or until we have examined
every possible matroid $M'$.
If \verb!Possibles! is not empty at this stage, we
know that, for every $k$ in \verb!Possibles!,
$(M,\texttt{Targets[k]})$ is a fascinating pair, so
we print out the information \verb!(15,r,i,k)!,
enabling us to identify the matroids in that pair.

Evaluating this command gives the following output:

{\small\begin{verbatim}
(15, 6, 445, 5)
(15, 6, 589, 19)
(15, 6, 5414, 16)
(15, 7, 0, 4)
(15, 7, 0, 8)
(15, 7, 34137, 22)
(15, 7, 34466, 22)
(15, 7, 34693, 22)
(15, 7, 34762, 22)
(15, 7, 34769, 22)
(15, 7, 35415, 23)
(15, 7, 35441, 23)
(15, 7, 35445, 23)
(15, 7, 35455, 23)
CPU time: 77608.72 s,  Wall time: 77606.78 s
\end{verbatim}}

We now seek to find the fascinating pairs, $(M,N)$,
where $|E(M)|=14$.
This implies that $|E(N)|=10$.
Since the only matroids in the list
\verb!Targets! with ground sets of cardinality $10$
are those with indices in $\{0,1,\ldots,5\}$, we
change the code so that \verb!Possibles! is
initially set to $\{0,1,\ldots,5\}$.
We also let \texttt{size} range between $13$ and $11$,
instead of $14$ and $11$.

{\small\begin{verbatim}
%time
for r in range(8):
    for i in range(len(IFCCatalogue[14][r])):
        M=IFCCatalogue[14][r][i]
        Possibles=set(range(6))
        for k in range(6):
            if not M.has_minor(Targets[k]):
                Possibles.discard(k)
        for size in range(13,10,-1):
            if len(Possibles)==0:
                break
            for rank in range(r,max(r-(14-size),0)-1,-1):
                if len(Possibles)==0:
                    break
                for j in range(len(IFCCatalogue[size][rank])):
                    Test=False
                    for k in Possibles:
                        if TargetMinors[size][rank][j][k]==1:
                            Test=True
                            break
                    if Test:
                        if M.has_minor(IFCCatalogue[size][rank][j]):
                            for k in range(6):
                                if TargetMinors[size][rank][j][k]==1:
                                    Possibles.discard(k)
                            if len(Possibles)==0:
                                break
        for k in Possibles:
            print (14,r,i,k)
\end{verbatim}}

This code produces the following output.

{\small\begin{verbatim}
(14, 6, 10, 4)
(14, 6, 14, 4)
(14, 6, 27, 5)
(14, 6, 1079, 4)
(14, 6, 1155, 4)
(14, 6, 1169, 4)
(14, 6, 1387, 1)
(14, 7, 3148, 4)
(14, 7, 3291, 0)
(14, 7, 3378, 5)
(14, 7, 3381, 1)
CPU time: 1073.44 s,  Wall time: 1072.64 s
\end{verbatim}}

This procedure has discovered $11$ pairs, but
the following commands show that, up to duality, there actually
only $9$ pairs.

{\small\begin{verbatim}
print IFCCatalogue[14][7][3148].
                is_isomorphic(IFCCatalogue[14][7][3291].dual())
print Targets[4].is_isomorphic(Targets[0].dual())
    True
    True
\end{verbatim}}

{\small\begin{verbatim}
print IFCCatalogue[14][7][3378].
                is_isomorphic(IFCCatalogue[14][7][3381].dual())
print Targets[5].is_isomorphic(Targets[1].dual())
    True
    True
\end{verbatim}}

Therefore, amongst fascinating pairs, $(M,N)$, with
$|E(M)|\leq 15$, there are, up to duality,
two containing $M(K_{4})$,
two containing $F_{7}$, and
four containing $M(K_{3,3})$.
The computer search has found an additional $23$ pairs.
All we need now do is confirm that these $23$ pairs are
as described in \Cref{sect2}, and that hence there are
no counterexamples to \Cref{maintheorem}.

The matroids \verb!Targets[1]! and \verb!Targets[5]! are
isomorphic to $M(K_{5})$ and $M^{*}(K_{5})$ respectively.

{\small\begin{verbatim}
print Targets[1].is_isomorphic(K5)
print Targets[5].is_isomorphic(K5.dual())
    True
    True
\end{verbatim}}

The tuples $(14, 6, 1387, 1)$,
$(14, 7, 3381, 1)$,
$(14, 6, 27, 5)$, and
$(15, 6, 445, 5)$, correspond to the pairs
$(M(\qml{7}),M(K_{5}))$,
$(M(Q_{3}^{\times}),M(K_{5}))$,
$(M^{*}(Y_{9}),M^{*}(K_{5}))$,
and $(M^{*}(\cml{10}),M^{*}(K_{5}))$.

{\small\begin{verbatim}
Q8cross=Matroid(graph=Graph({0:[1,5,7],1:[2,4],2:[3,6,7],
                                   3:[4,6,7],4:[5],5:[6],6:[7]}))
Y9=Matroid(graph=Graph({0:[1,3,5,7],1:[2,8],2:[3,6],
                                   3:[4],4:[5,8],5:[6],6:[7],7:[8]}))
CML10=Matroid(graph=graphs.CirculantGraph(10,[1,5]))

print IFCCatalogue[14][6][1387].is_isomorphic(QML7)
print IFCCatalogue[14][7][3381].is_isomorphic(Q8cross)
print IFCCatalogue[14][6][27].is_isomorphic(Y9.dual())
print IFCCatalogue[15][6][445].is_isomorphic(CML10.dual())
    True
    True
    True
    True
\end{verbatim}}

\verb!Targets[0]! and \verb!Targets[4]! are isomorphic to
$\Delta_{4}$ and $\Delta_{4}^{*}$ respectively.

{\small\begin{verbatim}
A=Matrix(GF(2),[
[1,0,0,1,1,1],
[1,1,0,0,1,0],
[0,1,1,0,1,1],
[0,0,1,1,1,0]])
Delta4=Matroid(reduced_matrix=A)
print Targets[0].is_isomorphic(Delta4)
print Targets[4].is_isomorphic(Delta4.dual())
    True
    True
\end{verbatim}}

The tuples
$(14, 6, 10, 4)$,
$(14, 6, 1079, 4)$,
$(14, 6, 14, 4)$,
$(14, 6, 1155, 4)$, and
$(14, 6, 1169, 4)$ correspond to the
fascinating pairs
$(A_{1}^{*},\Delta_{4}^{*})$,
$(A_{2}^{*},\Delta_{4}^{*})$,
$(A_{3}^{*},\Delta_{4}^{*})$,
$(A_{4}^{*},\Delta_{4}^{*})$, and
$(A_{5}^{*},\Delta_{4}^{*})$.

{\small\begin{verbatim}
print IFCCatalogue[14][6][10].is_isomorphic(
    Matroid(reduced_matrix=A.stack(Matrix(GF(2),[
    [1,0,1,0,0,1],
    [1,0,0,0,1,0],
    [0,1,1,0,1,0],
    [0,0,0,1,1,0]]))).dual())
print IFCCatalogue[14][6][1079].is_isomorphic(
    Matroid(reduced_matrix=A.stack(Matrix(GF(2),[
    [1,0,1,0,0,1],
    [1,0,0,0,1,0],
    [0,1,0,1,0,0],
    [0,1,0,0,1,1]]))).dual())
print IFCCatalogue[14][6][14].is_isomorphic(
    Matroid(reduced_matrix=A.stack(Matrix(GF(2),[
    [1,0,1,0,0,0],
    [1,0,0,0,1,0],
    [0,1,0,1,0,1],
    [0,0,0,1,1,0]]))).dual())
print IFCCatalogue[14][6][1155].is_isomorphic(
    Matroid(reduced_matrix=A.stack(Matrix(GF(2),[
    [1,0,1,0,0,0],
    [1,0,0,0,1,0],
    [0,1,0,1,0,0],
    [0,0,0,1,1,0]]))).dual())
print IFCCatalogue[14][6][1169].is_isomorphic(
    Matroid(reduced_matrix=A.stack(Matrix(GF(2),[
    [1,0,0,0,1,0],
    [0,1,1,0,1,0],
    [0,1,0,1,0,1],
    [0,0,1,1,0,0]]))).dual())
    True
    True
    True
    True
    True
\end{verbatim}}

The tuples $(14,7,3291,0)$ and $(15,7,0,4)$
correspond to $(A_{6},\Delta_{4})$ and
$(\Upsilon_{8}^{*},\Delta_{4}^{*})$.

{\small\begin{verbatim}
print IFCCatalogue[14][7][3291].is_isomorphic(
    Matroid(reduced_matrix=A.stack(Matrix(GF(2),[
    [1,1,0,0,1,0],
    [1,0,0,0,1,0],
    [1,0,0,0,0,0]])).augment(vector(GF(2),[1,1,0,0,0,0,1]))))
    True
\end{verbatim}}

{\small\begin{verbatim}
Upsilon8=Matroid(reduced_matrix=Matrix(GF(2),[
   [1,0,0,0,0,0,1,1],
   [1,1,0,0,0,0,0,1],
   [0,1,1,0,0,0,0,1],
   [0,0,1,1,0,0,0,1],
   [0,0,0,1,1,0,0,1],
   [0,0,0,0,1,1,0,1],
   [0,0,0,0,0,1,1,1]])).dual()
IFCCatalogue[15][7][0].is_isomorphic(Upsilon8.dual())
   True
\end{verbatim}}

\verb!Targets[22]! is isomorphic to $P^{*}$.

{\small\begin{verbatim}
B=Matrix(GF(2),[
[1,1,1,0,0,0,1],
[1,0,0,1,1,0,1],
[0,1,0,1,0,1,1],
[0,0,1,0,1,1,1]])
P=Matroid(reduced_matrix=B)
Targets[22].is_isomorphic(P.dual())
   True
\end{verbatim}}

The tuples
$(15, 7, 34466, 22)$,
$(15, 7, 34762, 22)$,
$(15, 7, 34137, 22)$,
$(15, 7, 34693, 22)$, and
$(15, 7, 34769, 22)$ correspond to the fascinating pairs
$(B_{1}^{*},P^{*})$,
$(B_{2}^{*},P^{*})$,
$(B_{3}^{*},P^{*})$,
$(B_{4}^{*},P^{*})$, and
$(B_{5}^{*},P^{*})$.

{\small\begin{verbatim}
print IFCCatalogue[15][7][34466].is_isomorphic(
    Matroid(reduced_matrix=B.stack(Matrix(GF(2),[
    [1,1,0,0,0,0,1],
    [1,0,1,1,0,1,0],
    [0,1,0,0,0,1,0],
    [0,0,0,1,1,0,0]]))).dual())
print IFCCatalogue[15][7][34762].is_isomorphic(
    Matroid(reduced_matrix=B.stack(Matrix(GF(2),[
    [1,0,1,1,0,1,0],
    [1,0,0,1,0,0,1],
    [0,1,0,0,0,0,1],
    [0,0,0,0,1,1,1]]))).dual())
print IFCCatalogue[15][7][34137].is_isomorphic(
    Matroid(reduced_matrix=B.stack(Matrix(GF(2),[
    [1,0,1,1,0,1,0],
    [1,0,0,0,0,0,1],
    [0,1,0,0,1,0,0],
    [0,0,1,0,1,0,1]]))).dual())
print IFCCatalogue[15][7][34693].is_isomorphic(
    Matroid(reduced_matrix=B.stack(Matrix(GF(2),[
    [1,0,1,1,0,1,0],
    [0,1,1,0,0,0,1],
    [0,1,0,1,0,0,0],
    [0,0,0,0,1,0,1]]))).dual())
print IFCCatalogue[15][7][34769].is_isomorphic(
    Matroid(reduced_matrix=B.stack(Matrix(GF(2),[
    [1,0,1,0,0,0,1],
    [1,0,0,1,0,0,1],
    [0,1,0,0,0,1,1],
    [0,0,0,0,1,1,1]]))).dual())
    True
    True
    True
    True
    True
\end{verbatim}}

\verb!Targets[23]! is isomorphic to $Q^{*}$.

{\small\begin{verbatim}
C=Matrix(GF(2),[
[1,0,0,1,1,1,0],
[1,1,0,0,1,0,1],
[0,1,1,0,1,1,1],
[0,0,1,1,1,0,1]])
Q=Matroid(reduced_matrix=C)
Targets[23].is_isomorphic(Q.dual())
    True
\end{verbatim}}

The tuples
$(15, 7, 35445, 23)$,
$(15, 7, 35441, 23)$,
$(15, 7, 35455, 23)$, and
$(15, 7, 35415, 23)$ correspond to the
fascinating pairs
$(C_{1}^{*},Q^{*})$,
$(C_{2}^{*},Q^{*})$,
$(C_{3}^{*},Q^{*})$, and
$(C_{4}^{*},Q^{*})$.

{\small\begin{verbatim}
print IFCCatalogue[15][7][35455].is_isomorphic(
    Matroid(reduced_matrix=C.stack(Matrix(GF(2),[
    [1,1,1,1,0,0,1],
    [1,1,1,0,1,0,1],
    [0,0,0,1,0,1,0],
    [0,0,0,0,1,0,1]]))).dual())
print IFCCatalogue[15][7][35441].is_isomorphic(
    Matroid(reduced_matrix=C.stack(Matrix(GF(2),[
    [1,1,1,1,0,0,1],
    [1,0,0,0,0,1,0],
    [0,1,0,0,1,0,0],
    [0,0,1,0,1,0,1]]))).dual())
print IFCCatalogue[15][7][35445].is_isomorphic(
    Matroid(reduced_matrix=C.stack(Matrix(GF(2),[
    [1,1,0,0,1,0,0],
    [1,0,0,0,0,1,0],
    [0,1,1,1,1,0,1],
    [0,0,1,0,1,0,0]]))).dual())
print IFCCatalogue[15][7][35415].is_isomorphic(
    Matroid(reduced_matrix=C.stack(Matrix(GF(2),[
    [1,0,0,0,1,1,0],
    [0,1,0,0,1,0,1],
    [0,0,1,0,1,1,1],
    [0,0,0,1,1,0,1]]))).dual())
    True
    True
    True
    True
\end{verbatim}}

The tuple $(15, 6, 589, 19)$ corresponds to
$(R^{*},D_{1}^{*})$.

{\small\begin{verbatim}
D=Matrix(GF(2),[
[1,1,1,0,0,0],
[1,0,0,1,1,0],
[0,1,0,1,0,1],
[0,0,1,0,1,1],
[1,0,0,0,0,1],
])
R=Matroid(reduced_matrix=D)
print Targets[19].is_isomorphic(R.dual())
print IFCCatalogue[15][6][589].is_isomorphic(
    Matroid(reduced_matrix=D.stack(Matrix(GF(2),[
    [1,1,0,0,0,0],
    [0,0,1,0,0,1],
    [0,0,0,1,0,1],
    [1,0,0,0,1,0]]))).dual())
    True
    True
\end{verbatim}}

The tuple $(15,6,5414,16)$ corresponds to
$(E_{1}^{*},S^{*})$.

{\small\begin{verbatim}
E=Matrix(GF(2),[
[1,0,0,1,1,1],
[1,1,0,0,1,0],
[0,1,1,0,1,1],
[0,0,1,1,1,0],
[1,0,0,1,0,1],
])
S=Matroid(reduced_matrix=E)
print Targets[16].is_isomorphic(S.dual())
print IFCCatalogue[15][6][5414].is_isomorphic(
    Matroid(reduced_matrix=F.stack(Matrix(GF(2),[
    [1,0,0,0,1,0],
    [1,0,0,0,0,1],
    [0,1,0,1,0,1],
    [0,0,1,1,0,0]]))).dual())
    True
    True
\end{verbatim}}

Finally, the tuple $(15,7,0,8)$ gives us the pair
$(\Upsilon_{8}^{*},\Upsilon_{6}^{*})$.

{\small\begin{verbatim}
print IFCCatalogue[15][7][0].is_isomorphic(Upsilon8.dual())
print Targets[8].is_isomorphic(Upsilon6.dual())
    True
    True
\end{verbatim}}

Now we can conclude there are no counterexamples, so
\Cref{maintheorem} holds.

We have characterised all fascinating pairs
satisfying $|E(M)|\leq 15$.
From this characterisation, it is straightforward to find all
interesting pairs satisfying the same constraint.
Certainly every fascinating pair is an interesting pair.
If $(M,M_{0})$ is interesting but not fascinating, then
there is an \ifc\ matroid, $M_{1}$,
satisfying $M_{0}\prec M_{1}\prec M$.
Now $(M,M_{1})$ is an interesting pair, so we can repeat this
argument and deduce that either $(M,M_{1})$ is fascinating,
or there is an \ifc\ matroid, $M_{2}$,
satisfying $M_{1}\prec M_{2}\prec M$.
Continuing in this way, we see that if $(M,M_{0})$ is interesting
but not fascinating, then $M_{0}\prec N\prec M$ for some \ifc\
matroid, $N$, such that $(M,N)$ is a fascinating pair.

This observation gives us our strategy for finding all interesting pairs.
Let $(M,N)$ range over all (up to duality) fascinating pairs with
$|E(M)|\leq 15$.
Consider each matroid, $T$, from the catalogue of \ifc\
matroids, that could potentially be a proper minor of $N$.
If $N$ has a proper $T$\dash minor, then
test to see whether any proper minor of $M$ produced by
deleting and contracting at most three elements
is \ifc\ with a $T$\dash minor.
If not, then $(M,T)$ is an interesting pair.
The following code performs exactly such a check.

{\small\begin{verbatim}
def GenerateInteresting(M,N):
    for n in range(6,len(N)):
        for r in range(N.rank()-(len(N)-n),N.rank()+1):
            for i in range(len(IFCCatalogue[n][r])):
                T=IFCCatalogue[n][r][i]
                if N.has_minor(T):
                    Between=False
                    for p in range(min(M.rank()-T.rank(),3)+1):
                        if Between:
                            break
                        for F in M.flats(p):
                            if Between:
                                break
                            if len(F)<4 and
                                        M.contract(F).has_minor(T):
                                if p==0:
                                    Lower=1
                                else:
                                    Lower=0
                                DeleteSet=
                                        M.groundset().difference(F)
                                for q in range(Lower,4-len(F)):
                                    if Between:
                                        break
                                    for D in Subsets(DeleteSet,q):
                                        Test=M.contract(F).delete(D)
                                        if Test.has_minor(T):
                                            if Test.is_3connected()
                                                    and IsIFC(Test):
                                                Between=True
                                                break
                    if not Between:
                        print (n,r,i)
\end{verbatim}}

Now we apply this function to all of the interesting pairs we
have discovered.

{\small\begin{verbatim}
GenerateInteresting(K5,K4)
GenerateInteresting(Q3,K4)
GenerateInteresting(Upsilon7,F7)
GenerateInteresting(Upsilon6.dual(),F7)
GenerateInteresting(H1,K33)
GenerateInteresting(H2,K33)
GenerateInteresting(H3,K33)
GenerateInteresting(QML7,K33)
    (6, 3, 0)
\end{verbatim}}

From this we see that $(M(\qml{7}),M(K_{4}))$ is an
interesting pair.
Now we check all the fascinating pairs discovered
by in the search.

{\small\begin{verbatim}
GenerateInteresting(IFCCatalogue[15][6][445],Targets[5])
GenerateInteresting(IFCCatalogue[15][6][589],Targets[19])
GenerateInteresting(IFCCatalogue[15][6][5414],Targets[16])
GenerateInteresting(IFCCatalogue[15][7][0],Targets[4])
    (7, 3, 0)
    (7, 4, 0)
\end{verbatim}}

This shows that $(\Upsilon_{8}^{*},F_{7})$ and $(\Upsilon_{8}^{*},F_{7}^{*})$
are interesting pairs.

{\small\begin{verbatim}
GenerateInteresting(IFCCatalogue[15][7][0],Targets[8])
    (7, 3, 0)
    (7, 4, 0)
\end{verbatim}}

This again tells us that
$(\Upsilon_{8}^{*},F_{7})$ and $(\Upsilon_{8}^{*},F_{7}^{*})$
are interesting pairs.
(This time we have discovered $F_{7}$ and $F_{7}^{*}$ by
searching for minors of $\Upsilon_{6}^{*}$, instead of $\Delta_{4}^{*}$.)

{\small\begin{verbatim}
GenerateInteresting(IFCCatalogue[15][7][34137],Targets[22])
GenerateInteresting(IFCCatalogue[15][7][34466],Targets[22])
GenerateInteresting(IFCCatalogue[15][7][34693],Targets[22])
GenerateInteresting(IFCCatalogue[15][7][34762],Targets[22])
GenerateInteresting(IFCCatalogue[15][7][34769],Targets[22])
GenerateInteresting(IFCCatalogue[15][7][35415],Targets[23])
GenerateInteresting(IFCCatalogue[15][7][35441],Targets[23])
GenerateInteresting(IFCCatalogue[15][7][35445],Targets[23])
GenerateInteresting(IFCCatalogue[15][7][35455],Targets[23])
GenerateInteresting(IFCCatalogue[14][6][10],Targets[4])
GenerateInteresting(IFCCatalogue[14][6][14],Targets[4])
GenerateInteresting(IFCCatalogue[14][6][27],Targets[5])
GenerateInteresting(IFCCatalogue[14][6][1079],Targets[4])
GenerateInteresting(IFCCatalogue[14][6][1155],Targets[4])
GenerateInteresting(IFCCatalogue[14][6][1169],Targets[4])
GenerateInteresting(IFCCatalogue[14][6][1387],Targets[1])
    (6, 3, 0)
\end{verbatim}}

As \verb!IFCCatalogue[14][6][1387])! is $M(\qml{7})$ we have again
discovered that $(M(\qml{7}),M(K_{4}))$ is an interesting pair.
We finish the search with the following commands.

{\small\begin{verbatim}
GenerateInteresting(IFCCatalogue[14][7][3148],Targets[4])
GenerateInteresting(IFCCatalogue[14][7][3291],Targets[0])
GenerateInteresting(IFCCatalogue[14][7][3378],Targets[5])
GenerateInteresting(IFCCatalogue[14][7][3381],Targets[1])
\end{verbatim}}

Thus the only interesting pairs that fail to be fascinating are
(up to duality) $(M(\qml{7}),M(K_{4}))$, $(\Upsilon_{8},F_{7})$, and
$(\Upsilon_{8}^{*},F_{7})$, and \Cref{mtheorem2} is now proved.


\end{document}